\documentclass{amsart}
\pdfoutput=1
\usepackage[margin=1.5in]{geometry}
\usepackage{amsmath}
\usepackage{amsfonts}
\usepackage{amsthm}
\usepackage[style=ext-alphabetic, maxnames=5, maxalphanames=5]{biblatex}
\usepackage{mathrsfs}
\usepackage{enumerate}
\usepackage{dsfont}
\usepackage{amssymb}
\usepackage{pifont}
\usepackage{mathtools}
\usepackage{import}
\usepackage{graphicx}
\usepackage{color}
\usepackage[dvipsnames, hyperref]{xcolor}
\usepackage{thmtools}
\usepackage{hyperref, cleveref}
\usepackage{import}
\usepackage{subfiles}
\usepackage{microtype}
\usepackage{svg}
\usepackage{calc}
\usepackage[final]{showkeys}

\hypersetup{
  colorlinks=true,
  linkcolor=Blue,
  citecolor=Blue
}
  
\newcommand{\parbdry}{\ensuremath{\dee_{\mathrm{par}}}}
\newcommand{\conbdry}{\ensuremath{\dee_{\mathrm{con}}}}

\newcommand{\maxdomain}{\ensuremath{{\Omega_{\mathrm{max}}}}}
\newcommand{\projflags}{\ensuremath{\mathcal{F}}}
\newcommand{\symflags}{\ensuremath{\mathcal{F}}}

\newcommand{\front}{\ensuremath{\mathrm{Fr}}}

\newcommand{\bgamh}{\ensuremath{\partial(\Gamma, \mathcal{H})}}
\newcommand{\verts}[1]{\ensuremath{\mathcal{V}_{#1}}}
\newcommand{\rface}[1]{\ensuremath{\tilde{#1}}}
\newcommand{\Homvf}{\ensuremath{\mathrm{Hom}}_{\mathrm{VF}}}

\newcommand{\hyp}[1]{[\ker #1]}

\newcommand{\minus}{-}

\declaretheorem[numberwithin=section]{theorem}
\declaretheorem[sibling=theorem, name=Lemma]{lem}
\declaretheorem[sibling=theorem, name=Proposition]{prop}

\declaretheorem[sibling=theorem, name=Corollary]{cor}

\declaretheorem[sibling=theorem, style=definition]{fact}
\declaretheorem[sibling=theorem, style=definition,
name=Definition]{definition}

\declaretheorem[sibling=theorem, style=definition]{remark}

\declaretheorem[sibling=theorem, style=definition]{notation}

\newcommand{\eps}{\varepsilon}
\newcommand{\dee}{\ensuremath{\partial}}

\newcommand{\mf}[1]{\ensuremath{\mathfrak{#1}}}
\newcommand{\mc}[1]{\ensuremath{\mathcal{#1}}}

\newcommand{\mr}[1]{\ensuremath{\mathrm{#1}}}
\newcommand{\ten}{\otimes}

\newcommand{\N}{\ensuremath{\mathbb{N}}}
\newcommand{\Z}{\ensuremath{\mathbb{Z}}}

\newcommand{\R}{\ensuremath{\mathbb{R}}}
\newcommand{\C}{\ensuremath{\mathbb{C}}}

\newcommand{\A}{\ensuremath{\mathbb{A}}}

\renewcommand{\H}{\ensuremath{\mathbb{H}}}
\renewcommand{\P}{\ensuremath{\mathbb{P}}}

\newcommand{\opp}{\ensuremath{\mathrm{opp}}}

\DeclareMathOperator{\Aut}{Aut}
\DeclareMathOperator{\codim}{codim}

\DeclareMathOperator{\Hom}{Hom}

\DeclareMathOperator{\Opp}{Opp}

\DeclareMathOperator{\Sym}{Sym}

\DeclareMathOperator{\Gr}{Gr}

\DeclareMathOperator{\supp}{supp}
\DeclareMathOperator{\spn}{span}
\DeclareMathOperator{\Stab}{Stab}
\DeclareMathOperator{\PSL}{PSL}
\DeclareMathOperator{\SL}{SL}
\DeclareMathOperator{\GL}{GL}
\DeclareMathOperator{\PGL}{PGL}

\title{Examples of extended geometrically finite representations}

\author{Theodore Weisman}

\address{Department of Mathematics, University of Michigan, Ann Arbor
  MI 48109, USA}

\email{tjwei@umich.edu}

\date{\today}

\begin{document}

\begin{abstract}
  This is the second of a pair of papers on extended geometrically
  finite (EGF) representations, which were originally posted as a
  single article under the title ``An extended definition of Anosov
  representation for relatively hyperbolic groups.'' In this paper, we
  prove that the holonomy representation of a projectively convex
  cocompact manifold with relatively hyperbolic fundamental group is
  always an EGF representation. We also prove that EGF representations
  arise as holonomy representations of convex projective manifolds
  with generalized cusps and as compositions of projectively convex
  cocompact representations with symmetric representations of
  $\mathrm{SL}(d, \mathbb{R})$. We additionally show that any small
  deformation of a representation of the latter form is still EGF.
\end{abstract}

\maketitle

\tableofcontents

\section{Introduction}

\subsection{Format of the paper}

This paper constitutes the second part of a preprint originally posted
in April 2022 under the title \emph{An extended definition of Anosov
  representation for relatively hyperbolic groups}
\cite{Weisman2022}. The first part of the original preprint (which
introduces EGF representations and proves a relative stability result)
can still be found with its original title at the original arXiv
posting.

The division was made largely for the sake of decreasing the total
length of the paper, so the contents of this article are essentially
unchanged from the way they appeared in earlier versions of
\cite{Weisman2022}.

\subsection{Overview}

In the last two decades, \emph{Anosov representations} have emerged as
important objects in the study of discrete subgroups of semisimple Lie
groups. Originally defined for surface groups by Labourie in
\cite{labourie2006anosov}, and extended to arbitrary word-hyperbolic
groups by Guichard-Wienhard \cite{gw2012anosov}, Anosov
representations generalize geometric and dynamical properties of
convex cocompact representations in rank one: an Anosov representation
is always a quasi-isometric embedding $\Gamma \to G$ for a Gromov
hyperbolic group $\Gamma$ and a semisimple Lie group $G$, and comes
equipped with an equivariant boundary embedding $\dee \Gamma \to G/P$,
where $P \subset G$ is a parabolic subgroup.

Anosov representations have come to be accepted as a suitable
higher-rank generalization of convex cocompactness, which raises the
question of whether there is also an analogous generalization of
geometrical finiteness. Several authors (see
\cite{kl2018relativizing}, \cite{zhu2019relatively},
\cite{czz2021cusped}, \cite{zz1}) have previously defined notions of
\emph{relative Anosov} representations, but none of the proposed
definitions capture certain natural families of examples of
``geometrically finite'' behavior in higher rank.

In \cite{Weisman2022}, we introduced a new class of representations of
relatively hyperbolic groups into semisimple Lie groups, called
\emph{extended geometrically finite (EGF)} representations. EGF
representations generalize all existing definitions of relative Anosov
representations. In addition, the definition is flexible enough to
cover many additional examples of higher-rank ``geometrically finite''
behavior, and to allow for EGF representations to deform in ways not
available to relative Anosov representations. Specifically, EGF
representations satisfy the following theorem:
\begin{theorem}[{See \cite[Theorem 1.4]{Weisman2022}}]
  \label{thm:stability_theorem}
  Let $(\Gamma, \mc{H})$ be a relatively hyperbolic pair, and let
  $\rho:\Gamma \to \PGL(d, \R)$ be an extended geometrically finite
  representation, with boundary extension $\phi$. If
  $\mc{W} \subset \Hom(\Gamma, \PGL(d,\R))$ is a subspace which is
  \emph{peripherally stable} with respect to $\rho, \phi$, then an
  open subset of $\mc{W}$ containing $\rho$ consists of EGF
  representations.
\end{theorem}

We will review the terminology in \Cref{thm:stability_theorem} in
\Cref{sec:egf_review}. The main aim of this article is to explain how
various examples fit into the theory of EGF representations, and
highlight applications of \Cref{thm:stability_theorem} for these
examples. In some cases, the theorem recovers known stability results
for various classes of discrete groups, and in other cases it yields
new results.

\subsection{Results}

We refer to \Cref{sec:egf_review} for the definition of an EGF
representation. See \cite{Weisman2022} for further detail.

\subsubsection{Convex cocompact groups in $\PGL(d, \R)$}
\label{sec:convex_cocompact_intro}

In \cite{dgk2017convex}, Danciger-Gu\'eritaud-Kassel introduced a
notion of \emph{convex cocompactness} for projective orbifolds,
i.e. orbifolds with a real projective structure. Roughly, a group
$\Gamma \subset \PGL(d, \R)$ is \emph{convex cocompact} if it acts
with compact quotient on a certain closed invariant convex subset
inside a \emph{properly convex domain} $\Omega \subset \P(\R^d)$. We
refer to \Cref{sec:convex_cocompact} for the precise definition.

Convex cocompact groups in $\PGL(d, \R)$ are closely related to Anosov
subgroups. Let $\{e_1, \ldots, e_d\}$ denote the standard basis for
$\R^{d}$, and let $P_{1,d-1}$ denote the parabolic subgroup
stabilizing the flag $(\spn\{e_1\}, \spn\{e_1, \ldots, e_{d-1}\})$. If
$\Gamma$ is a word-hyperbolic subgroup of $\PGL(d, \R)$ which is
convex cocompact in the sense of \cite{dgk2017convex}, then the
inclusion $\Gamma \hookrightarrow \PGL(d, \R)$ is a $P_{1,d-1}$-Anosov
representation. Conversely, any $P_{1,d-1}$-Anosov representation
$\rho:\Gamma \to \PGL(d, \R)$ which preserves some properly convex
domain in $\P(\R^d)$ is convex cocompact. Further, for any semisimple
Lie group $G$ and any parabolic subgroup $P \subset G$, one can find a
representation $\phi:G \to \PGL(d, \R)$ so that a representation
$\rho:\Gamma \to G$ of a hyperbolic group $\Gamma$ is $P$-Anosov if
and only if the composition $\phi \circ \rho$ has convex cocompact
image (see \cite{ggkw2017anosov}, \cite{zimmer2021projective}). Thus,
convex cocompactness in $\PGL(d, \R)$ can be used to give a definition
of Anosov representation in terms of convex projective geometry.

Unlike several other definitions of Anosov representations, however,
projective convex cocompactness immediately generalizes beyond the
realm of hyperbolic groups. Indeed, there are a number of different
constructions for non-hyperbolic convex cocompact subgroups of
$\PGL(d, \R)$; see e.g. \cite{benoist2006convexes},
\cite{bdl2015convex}, \cite{clm2020convex}, \cite{dgklm2021convex},
\cite{BV} (see also \cite[Sec. 2.6]{weisman2020dynamical} for an
overview, or the forthcoming work \cite{DGKexamples}). In each of the
examples we have mentioned here, the convex cocompact group $\Gamma$
is (abstractly) a relatively hyperbolic group.

These relatively hyperbolic convex cocompact groups do \emph{not} fit
into previously existing theories of relative Anosov representations
(see \cite[Remark 1.14]{weisman2020dynamical}). However, we show in
this paper that they do always give rise to EGF representations:
\begin{theorem}
  \label{thm:convex_cocompact_egf}
  Let $\Gamma$ be a convex cocompact subgroup of $\PGL(d, \R)$, and
  suppose that $\Gamma$ is relatively hyperbolic. Then the inclusion
  $\Gamma \hookrightarrow \PGL(d, \R)$ is extended geometrically
  finite.
\end{theorem}

Our proof of \Cref{thm:convex_cocompact_egf} builds on our earlier
study \cite{weisman2020dynamical} of ``boundary maps'' from the
Bowditch boundary of a relatively hyperbolic convex cocompact group
$\Gamma$ to a quotient of the boundary of a $\Gamma$-invariant domain
in $\P(\R^d)$, as well as related work of Islam-Zimmer on the same
topic \cite{iz2022structure}.

\begin{remark}
  Not every convex cocompact group in $\PGL(d, \R)$ is abstractly
  relatively hyperbolic. If $\Gamma$ is a uniform lattice in
  $\PSL(n, \R)$, then $\Gamma$ acts on the space $\tilde{X}$ of
  $n \times n$ positive-definite symmetric matrices, which embeds into
  the vector space $V$ of symmetric $n \times n$ matrices. The image
  of $\tilde{X}$ in the projective space $\P(V)$ is a properly convex
  open set $X$, and the induced action of $\Gamma$ on $X$ is properly
  discontinuous and cocompact, meaning that $\Gamma$ is convex
  cocompact in $\PGL(V)$. However, the group $\Gamma$ is not
  relatively hyperbolic whenever $n > 2$.

  On the other hand, there are no examples known of convex cocompact
  groups in $\PGL(d, \R)$ which are \emph{not} either relatively
  hyperbolic or isomorphic to a uniform lattice in some semisimple Lie
  group $G$.
\end{remark}

\subsubsection{Convex projective orbifolds with generalized cusps}

In \cite{clt2018deforming}, Cooper-Long-Tillmann studied a different
generalization of geometrical finiteness in the context of convex
projective geometry. They considered the case of a strictly convex
projective $(d-1)$-manifold (possibly with boundary) which decomposes
into a compact piece and finitely many \emph{generalized cusps}. In
the language of \cite{clt2018deforming}, a \emph{generalized cusp} is
a strictly convex projective $(d-1)$-manifold homemorphic to
$N \times [0, \infty)$, where $N$ is a closed $(d-2)$-manifold with
virtually nilpotent fundamental group.

Later, Ballas-Cooper-Leitner \cite{bcl2020generalized} classified
generalized cusps into $d$ different \emph{cusp types}, in particular
showing that a generalized cusp always has virtually abelian
fundamental group. A ``type $0$'' generalized cusp is projectively
equivalent to a hyperbolic cusp (and has virtually unipotent
holonomy), while a ``type $(d-1)$'' generalized cusp has virtually
diagonalizable holonomy. The other cusp types are ``interpolations''
between type $0$ and type $d-1$.

Now fix a strictly convex projective $(d-1)$-manifold $M$, and assume
that $M$ is a union of a compact manifold and finitely many
generalized cusps. The holonomy of $M$ is a relative Anosov
representation if and only if all of its cusps have type $0$ (see
\cite{zhu2019relatively}), and the image of the holonomy is convex
cocompact (in the sense mentioned previously) if and only if all of
its cusps have type $d-1$. On the other hand, the framework of EGF
representations still applies in the presence of all types of cusps:

\begin{theorem}
  \label{thm:generalized_cusps_exist}
  For every $d$ and for every $0 \le t < d-1$, there exists a convex
  projective $(d-1)$-manifold with a type $t$ generalized cusp whose
  holonomy is EGF.
\end{theorem}

The fact that for every $d$ and every $0 < t < d - 1$ there even
exists a convex projective $(d-1)$-manifold containing a type-$t$ cusp
is a theorem of Bobb \cite{bobb2019convex}; the content of
\Cref{thm:generalized_cusps_exist} is that Bobb's construction
produces examples of EGF representations. Our proof relies on the fact
that both EGF representations and holonomy representations of
manifolds with generalized cusps are \emph{relatively
  stable}. Specifically, for a relatively hyperbolic pair
$(\Gamma, \mc{H})$, we let
\[
  \Homvf(\Gamma, \PGL(d, \R), \mc{H})
\]
denote the space of \emph{virtual flag} representations of $\Gamma$
into $\PGL(d, \R)$: the space of representations
$\rho:\Gamma \to \PGL(d+1, \R)$ such that, for each $H \in \mc{H}$,
there is a finite-index subgroup $H' \subset H$ so that the
restriction $\rho|_{H'}$ is discrete faithful and $\rho(H')$ is
conjugate to a group of upper-triangular matrices.

Cooper-Long-Tillmann show (see \cite[Thm. 0.1]{clt2018deforming}) that
holonomy representations of convex projective manifolds with
generalized cusps form an open subset of
$\Homvf(\Gamma, \PGL(d, \R), \mc{H})$. We prove:
\begin{theorem}
  \label{thm:vf_deformations}
  Let $M = \Omega / \Gamma$ be a finite-volume convex projective
  $(d-1)$-manifold, and suppose that $\Omega$ is strictly convex (so
  that $\Gamma = \pi_1M$ is hyperbolic relative to the collection
  $\mc{H}$ of cusp groups, and the holonomy
  $\rho:\pi_1M \to \PGL(d, \R)$ is $1$-EGF with a boundary extension
  $\phi$).

  Then $\Homvf(\Gamma, \PGL(d, \R), \mc{H})$ is
  peripherally stable at $(\rho, \phi)$. In particular, due to
  \Cref{thm:stability_theorem}, an open subset of
  $\Homvf(\Gamma, \PGL(d, \R), \mc{H})$ consists of EGF
  representations.
\end{theorem}

\Cref{thm:vf_deformations} implies \Cref{thm:generalized_cusps_exist}
because known constructions of convex projective manifolds with
generalized cusps (see \cite{ballas2021}, \cite{bm2020properly},
\cite{bobb2019convex}) proceed by starting with a manifold
$M = \Omega / \Gamma$ as above, and then performing an (arbitrarily
small) deformation of the holonomy of $M$ inside of
$\Homvf(\pi_1M, \PGL(d, \R), \mc{H})$ to produce a new convex
projective structure on $M$ which realizes its ends as generalized
cusps.

\begin{remark}
  In \cite{cm2014finitude}, Crampon-Marquis defined several notions of
  geometrical finiteness for strictly convex projective manifolds, and
  claimed that their definitions were all equivalent. It appears that
  this was an error, and some of their definitions are actually
  stronger than others. None of their definitions allow for the
  presence of generalized cusps which are not ``type $0$'' in the
  Ballas-Cooper-Leitner classification.

  On the other hand, work of Cooper-Long-Tillmann \cite[Thm
  11.6]{clt2015convex} implies that any finite-volume strictly convex
  projective manifold $M = \Omega / \Gamma$ as in
  \Cref{thm:vf_deformations} only has cusps of type $0$, and is
  actually geometrically finite in the \emph{strongest} sense defined
  by Crampon-Marquis. Zhu \cite[Prop 8.7]{zhu2019relatively} proved
  that the holonomy representation of such a geometrically finite
  manifold is always relative Anosov (relative to the cusp groups),
  hence EGF by \cite[Thm. 1.10]{Weisman2022}. This justifies the
  assertion in the first paragraph of \Cref{thm:vf_deformations}.
\end{remark}

\subsubsection{Compositions with symmetric representations}

The last class of examples we consider in this paper also derive from
the convex cocompact representations discussed in
\Cref{sec:convex_cocompact_intro}. Suppose that $\Gamma$ is a
relatively hyperbolic group, relative to a collection $\mc{H}$ of
virtually abelian subgroups. Let $V$ be the vector space $\R^{d}$, and
let $\bar{\rho}:\Gamma \to \PGL(V)$ be a discrete faithful
representation whose image is a convex cocompact group in the sense of
Danciger-Gu\'eritaud-Kassel. Possibly after replacing $\Gamma$ with a
finite-index subgroup, we may lift $\bar{\rho}$ to a discrete faithful
representation $\rho:\Gamma \to \SL(V)$. Abusing terminology slightly,
we will refer to both $\bar{\rho}$ and $\rho$ as \emph{convex
  cocompact} representations.

If $\Gamma$ is a hyperbolic group, and $\rho$ is a $P$-Anosov
representation for some parabolic $P \subset \SL(V)$, then the
composition of $\rho$ with the symmetric representation
\[
  \tau_k:\SL(V) \to \SL(\Sym^k(V))
\]
is a new representation of $\Gamma$ which is $P'$-Anosov, for some
parabolic $P' \subset \SL(\Sym^k(V))$ depending only on $P$. Due to
the close connection between $P_{1,d-1}$-Anosov representations and
projectively convex cocompact representations, one might hope that
even when $\Gamma$ is \emph{not} hyperbolic, but $\rho$ is convex
cocompact, then the composition $\tau_k \circ \rho$ is still convex
cocompact.

It appears that this is not the case: we do \emph{not} expect the
composition $\tau_k \circ \rho$ to be convex cocompact unless $\Gamma$
is a hyperbolic group. However, we can still show:
\begin{theorem}
  \label{thm:symmetric_powers_egf}
  Let $\Gamma$ be hyperbolic relative to virtually abelian subgroups,
  and let $\rho:\Gamma \to \SL(V)$ be projectively convex
  cocompact. For any $k \ge 1$, the representation $\tau_k \circ \rho$
  is EGF, with respect to the parabolic $P' \subset \SL(\Sym^k(V))$
  stabilizing a line in a hyperplane.
\end{theorem}

In \cite{dgk2017convex}, Danciger-Gu\'eritaud-Kassel showed that
projectively convex cocompact representations are \emph{absolutely}
stable: if $\rho:\Gamma \to \SL(V)$ is projectively convex cocompact,
then any sufficiently small deformation of $\rho$ in
$\Hom(\Gamma, \SL(V))$ is also projectively convex cocompact (in
particular, it is discrete with finite kernel). The proof in
\cite{dgk2017convex} does \emph{not} apply to the representations in
\Cref{thm:symmetric_powers_egf}. However, we prove the following:
\begin{theorem}
  \label{thm:symmetric_powers_stable}
  Let $\rho, \Gamma$ be as in \Cref{thm:symmetric_powers_egf}, so that
  $\tau_k \circ \rho$ is an EGF representation for every $k \ge
  1$. Then for some boundary extension $\phi$ for $\tau_k \circ \rho$,
  the entire subspace $\Hom(\Gamma, \SL(\Sym^k(V))$ is peripherally
  stable about $\tau_k \circ \rho$, $\phi$.

  In particular, due to \Cref{thm:stability_theorem}, an open subset
  of $\Hom(\Gamma, \SL(\Sym^k(V)))$ containing $\tau_k \circ \rho$
  consists of EGF representations.
\end{theorem}

Thus, taking $k > 1$, \Cref{thm:symmetric_powers_stable} provides new
examples of discrete subgroups of higher-rank Lie groups which are
\emph{absolutely} stable in their representation varieties. (When
$k = 1$, the stability of these representations follows from the
stability theorem of Danciger-Gu\'eritaud-Kassel.)

\subsection{Acknowledgements}

The author would like to thank his PhD advisor, Jeff Danciger, for
encouragement and much-needed advice. Additional thanks are owed to
Daniel Allcock, Dick Canary, Fanny Kassel, Max Riestenberg, Feng Zhu,
and Andy Zimmer for helpful feedback. This work was supported in part
by NSF grants DMS-1937215 and DMS-2202770.

\section{Review of EGF representations}
\label{sec:egf_review}

In this section we briefly review some of the definitions,
terminology, and basic results surrounding extended geometrically
finite representations. Although EGF representations can be defined
with respect to any (symmetric) parabolic subgroup $P$ of a semisimple
Lie group $G$, in this paper we will only consider the theory in the
case where $G = \PGL(d, \R)$ or $\SL(d, \R)$, and $P$ is the
stabilizer of a flag of type $(1, d-1)$ in $\R^d$. We refer to
\cite{Weisman2022} for a more thorough introduction to the theory,
including some background regarding relatively hyperbolic groups and
semisimple Lie groups.

\subsection{Extended convergence group actions}

The definition of an EGF representation is based on a characterization
of Anosov representations in terms of topological
dynamics---specifically, in terms of \emph{convergence
  actions}. Recall that if $\Gamma$ is a group acting by
homeomorphisms on a Hausdorff space $M$, we say that $\Gamma$ is a
\emph{convergence group} and that the action is a \emph{convergence
  group action} if, for every divergent sequence
$\gamma_n \in \Gamma$, after extracting a subsequence, one can find
(not necessarily distinct) points $a, b \in M$ so that the
restrictions $\gamma_n|_{M \minus \{a\}}$ converge to the constant map
$b$, uniformly on compacts. The point $a$ can be thought of as an
``repelling point'' for the divergent sequence $\gamma_n$, and the
point $b$ can be thought of as an ``attracting point;'' the complement
of the ``repelling point'' is a ``basin of attraction'' for the
sequence $\gamma_n$. When $(\Gamma, \mc{H})$ is a relatively
hyperbolic pair, then $\Gamma$ acts as a convergence group on the
Bowditch boundary $\bgamh$.

In rank one, geometrically finite representations can be characterized
using convergence group actions of relatively hyperbolic groups. To
define \emph{extended} geometrically finite actions, we broaden the
definition in two ways simultaneously. Essentially, we no longer
require the ``basin of attraction'' for a divergent sequence to be the
complement of a singleton, and we no longer require the ``attracting
point'' to be a singleton either.

\begin{definition}
  \label{defn:extended_convergence_group}
  Let $(\Gamma, \mc{H})$ be a relatively hyperbolic pair, acting on a
  Hausdorff space $M$ by homeomorphisms. Let $\Lambda \subset M$ be a
  closed $\Gamma$-invariant set. We say that a $\Gamma$-equivariant
  surjective map $\phi:\Lambda \to \bgamh$ \emph{extends the
    convergence group action} of $\Gamma$ on $\bgamh$ if, for each
  $z \in \bgamh$, there exists an open set $C_z \subset M$ satisfying:
  \begin{enumerate}
  \item For every $z \in \bgamh$, we have
    $\Lambda \subset C_z \cup \phi^{-1}(z)$.
  \item For every sequence $\gamma_n \in \Gamma$ such that
    $\gamma_n \to z_+$ and $\gamma_n^{-1} \to z_-$ for
    $z_\pm \in \bgamh$, every compact subset $K \subset C_{z_-}$, and
    every open set $U \subset M$ containing $\phi^{-1}(z_+)$, we have
    $\gamma_n K \subset U$ for all sufficiently large $n$.
  \end{enumerate}
\end{definition}

\subsection{Linear actions on flags in $\R^d$}

An extended geometrically finite representation into $\PGL(d, \R)$ or
$\SL(d, \R)$ is essentially just an extended convergence action of a
relatively hyperbolic group on the space of \emph{flags} in $\R^d$,
with an extra condition taking into account some of the additional
structure on this flag space.

\begin{notation}
  When $V$ is a real vector space, we use $\Gr(k, V)$ to denote the
  Grassmannian of $k$-planes in $V$; recall that $\Gr(1, V)$ is the
  same as the projective space $\P(V)$. When $\dim(V) = d$, recall
  that there is also a canonical equivariant identification of
  $\Gr(d-1, V)$ with the \emph{dual} projective space $\P(V^*)$, since
  the projectivization of a nonzero linear functional in $V^*$ is
  determined by its kernel.

  We let $\projflags(V)$ denote the space of \emph{flags} of type
  $(1, \dim(V) - 1)$ in $V$, i.e. pairs of subspaces $(V_i, V_j)$ such
  that $V_i \subset V_j \subset V$ and $\dim(V_i) = \codim(V_j) =
  1$. We will write $\Gr(k, d)$ and $\projflags(d)$ as shorthand for
  $\Gr(k, \R^d)$ and $\projflags(\R^d)$, respectively.
\end{notation}

\begin{definition}
  We say that a pair of flags $F = (V, W), F' = (V', W')$ in
  $\projflags(d)$ are \emph{transverse} if $V \oplus W' = \R^d$ and
  $W \oplus V' = \R^d$.

  For any $1 \le k < d$, and any subspace $V_k \subset \Gr(k, d)$, we
  let $\Opp(\xi) \subset \Gr(d - k, d)$ denote the space of
  $(d-k)$-planes transverse to $V_k$. Similarly, for a fixed flag
  $\xi \in \projflags(d)$, we let $\Opp(\xi)$ denote the space of
  flags
  \[
    \{\nu \in \projflags : \nu \textrm{ is transverse to }\xi\}.
  \]
  This is always an open dense subset of $\projflags(d)$. For a subset
  $X \subset \projflags$, we also let $\Opp(X)$ denote the set
  \[
    \bigcap_{\xi \in X} \Opp(\xi).
  \]
\end{definition}

\begin{definition}
  Let $g_n$ be a sequence in $\PGL(d, \R)$ or $\SL(d, \R)$, and let
  $1 \le k < d$.
  \begin{enumerate}
  \item The sequence $g_n$ is called \emph{$k$-contracting} if there
    is a nonempty open subset $U \subset \Gr(k, d)$ such that $g_nU$
    converges to a singleton $\{\xi\}$. The point $\xi \in \Gr(k, d)$
    is called the \emph{$k$-limit} of the sequence $g_n$.
  \item The sequence $g_n$ is called \emph{$k$-divergent} if every
    subsequence of $g_n$ has a further subsequence which is
    $k$-contracting.
  \item A subgroup $\Gamma$ in $\PGL(d, \R)$ or $\SL(d, \R)$ is called
    \emph{$k$-divergent} if every sequence of pairwise distinct
    elements in $\Gamma$ is $k$-divergent.
  \end{enumerate}
\end{definition}

Any divergent sequence in $\PGL(d, \R)$ must be $k$-divergent for at
least one $k$ with $1 \le k < d$.  When $g_n$ is a $k$-divergent
sequence in $\PGL(d, \R)$, then the sequence of inverses $g_n^{-1}$ is
$(d-k)$-divergent. More precisely, we have the following fact, which
can be verified using the singular value decomposition of elements in
$\PGL(d, \R)$ (or for a general statement see \cite[Lemma
4.19]{klp2017anosov} and \cite[Appendix A]{Weisman2022}):
\begin{fact}
  \label{fact:contraction_symmetry}
  For a sequence $g_n \in \PGL(d, \R)$ and points
  $\xi_- \in \Gr(d-k, d)$ and $\xi_+ \in \Gr(k, d)$, the following are
  equivalent:
  \begin{enumerate}
  \item $g_n$ is $k$-contracting and $g_n|_{\Opp(\xi_-)} \to \xi_+$
    uniformly on compacts.
  \item $g_n$ is $k$-divergent, $g_n$ has unique $k$-limit point
    $\xi_+$, and $g_n^{-1}$ has unique $(d-k)$-limit point $\xi_-$.
  \end{enumerate}
\end{fact}

When a sequence $g_n$ in $\PGL(d, \R)$ or $\SL(d, \R)$ is both
$1$-contracting and $(d-1)$-contracting (resp. $1$-divergent and
$d$-divergent), we will say that it is \emph{$(1,d-1)$-contracting} or
\emph{$(1,d-1)$-divergent}. Owing to \Cref{fact:contraction_symmetry},
any $1$-divergent subgroup of $\PGL(d, \R)$ or $\SL(d, \R)$ is
automatically also $(d-1)$-divergent, and vice-versa, so we will also
call such subgroups $(1,d-1)$-divergent.

\begin{remark}
  The more usual definition of $(1,d-1)$-divergence (or more
  generally, \emph{$P$-divergence} for a parabolic subgroup $P$ in a
  semisimple Lie group $G$) is stated in terms of the behavior of the
  \emph{Cartan projection} of a sequence $g_n \in G$. We refer again
  to \cite[Lemma 4.19]{klp2017anosov} and \cite[Appendix
  A]{Weisman2022} for the equivalence.
\end{remark}

\subsection{EGF representations}

\begin{definition}
  Let $\Lambda \subset \projflags(d)$ be a subset, and let
  $\phi:\Lambda \to Z$ be a surjective map to some space $Z$. We say
  that $\phi$ is \emph{transverse} if for every pair of distinct
  points $z_1, z_2 \in Z$, every flag in $\phi^{-1}(z_1)$ is
  transverse to every flag in $\phi^{-1}(z_2)$.
\end{definition}

\begin{definition}
  \label{defn:egf}
  A representation $\rho:\Gamma \to \PGL(d, \R)$ (or to $\SL(d, \R))$
  is $1$-\emph{extended geometrically finite} or ($1$-EGF) if there
  exists a compact $\rho$-invariant subset
  $\Lambda \subset \projflags(d)$ and a surjective transverse map
  $\phi:\Lambda \to \bgamh$ extending the convergence action of
  $\Gamma$ on $\bgamh$.

  The map $\phi$ is called a \emph{boundary extension} of the
  representation $\rho$, and the set $\Lambda$ is called a
  \emph{boundary set}. Since we are only concerned with $1$-EGF
  representations in this paper, we will almost always just refer to
  $1$-EGF representations as \emph{EGF representations}.
\end{definition}

\begin{remark}
  In general there may be more than one possible choice for the
  boundary set and boundary extension associated to a given EGF
  representation; there might also be more than one possible choice
  for the open sets $C_z \subset \projflags(d)$ specified in
  \Cref{defn:extended_convergence_group}. It is also always possible
  to choose the boundary extension $\phi$ so that the preimage of any
  \emph{conical limit point} (see \Cref{defn:conical_limit} below) in
  $\bgamh$ is a singleton; see \cite[Proposition 4.8]{Weisman2022}.
\end{remark}

\subsection{An alternative characterization}

In order to use \Cref{defn:egf} to directly verify that a given
representation is EGF, we must consider the dynamical behavior of
(essentially) arbitrary divergent sequences in $\Gamma$, since any
divergent sequence $\gamma_n$ in a relatively hyperbolic group has a
subsequence which satisfies $\gamma_n \to z_+$ and
$\gamma_n^{-1} \to z_-$ for points $z_\pm \in \bgamh$. Often, we will
only want to consider sequences in $\Gamma$ which either limit to the
boundary of $\Gamma$ \emph{along geodesics} (i.e. \emph{conical limit
  sequences}), or sequences which diverge inside of a fixed peripheral
subgroup of $\Gamma$.

In \Cref{prop:conical_peripheral_implies_egf} below, we give an
alternative characterization of EGF representations which allows us to
restrict our attention to sequences of one of these two forms. We
first recall some terminology:

\begin{definition}
  \label{defn:conical_limit}
  Let $(\Gamma, \mc{H})$ be a relatively hyperbolic pair. We say that
  a sequence $\gamma_n \in \Gamma$ \emph{limits conically to} a point
  $z \in \bgamh$ if there exist \emph{distinct} points
  $a,b \in \bgamh$ such that $\gamma_n^{-1}z \to a$ and
  $\gamma_n^{-1}y \to b$ for every $y \ne z$. If a sequence
  $\gamma_n \in \Gamma$ limits conically to $z$, then we say that $z$
  is a \emph{conical limit point}.
\end{definition}

Every point in the Bowditch boundary of a relatively hyperbolic pair
$(\Gamma, \mc{H})$ is either a conical limit point or a
\emph{parabolic point}, i.e. the (unique) fixed point of a peripheral
subgroup in $\mc{H}$. When $p$ is a parabolic point, we let
$\Gamma_p \in \mc{H}$ denote the peripheral subgroup stabilizing $p$.

We can now state our alternative characterization of EGF
representations:

\begin{prop}[{See \cite[Proposition 4.6]{Weisman2022}}]
  \label{prop:conical_peripheral_implies_egf}
  Let $\rho:\Gamma \to \PGL(d, \R)$ be a representation of a
  relatively hyperbolic group, and let $\Lambda \subset \projflags(d)$
  be a closed $\rho(\Gamma)$-invariant set. Suppose that
  $\phi:\Lambda \to \bgamh$ is a continuous surjective
  $\rho$-equivariant transverse map.

  Then $\rho$ is an EGF representation with EGF boundary extension
  $\phi$ if and only if both of the following conditions hold:
  \begin{enumerate}
  \item \label{item:conical_convergence} For any sequence
    $\gamma_n \in \Gamma$ limiting conically to some point in
    $\bgamh$, the sequences $\rho(\gamma_n^{\pm 1})$ are
    $(1,d-1)$-divergent, and every $(1,d-1)$-limit point of
    $\rho(\gamma_n^{\pm 1})$ lies in $\Lambda$.
  \item \label{item:parabolic_convergence} For every parabolic point
    $p \in \bgamh$, there exists an open set $C_p \subset \projflags$,
    with $\Lambda \subset C_p \cup \phi^{-1}(p)$, such that for any
    compact $K \subset C_p$ and any open set $U$ containing
    $\phi^{-1}(p)$, for all but finitely many $\gamma \in \Gamma_p$,
    we have $\rho(\gamma) \cdot K \subset U$.
  \end{enumerate}
\end{prop}

\subsection{Peripheral stability} The central result of
\cite{Weisman2022} is that EGF representations are relatively stable:
any sufficiently small deformation which satisfies a dynamical
condition on peripheral subgroups is still EGF.

For an EGF representation $\rho$ with boundary extension $\phi$, if
$p$ is a parabolic point in $\bgamh$, then the set $\phi^{-1}(p)$ can
be thought of as \emph{both} an ``attracting set'' and a ``repelling
set'' for the action of $\Gamma_p$ on $\projflags(d)$. More precisely,
if $K$ is any compact set in $C_p$, and $U$ is any neighborhood of
$\phi^{-1}(p)$, there is always some finite set $F$ such that
$\rho(\Gamma_p \minus F)K \subset U$.

We want to consider deformations of the representation $\rho$ which
satisfy the property that some set close to $\phi^{-1}(p)$ is still an
``attracting set'' for $\Gamma_p$ with ``basin of attraction''
$C_p$. We also want to ask for the ``strength'' of the attraction to
not decrease too much, which is quantified by the compact set
$K \subset C_p$, the open set $U \supset \phi^{-1}(p)$, and the finite
set $F$ mentioned previously.

\begin{definition}[Peripheral stability]
  Let $(\Gamma, \mc{H})$ be a relatively hyperbolic pair, and let
  $\rho:\Gamma \to \PGL(d, \R)$ be an EGF representation with boundary
  extension $\phi:\Lambda \to \bgamh$. We say that a subspace
  $\mc{W} \subset \Hom(\Gamma, \PGL(d, \R))$ is \emph{peripherally
    stable} (with respect to the data $(\rho, \phi)$) if, for every
  parabolic point $p \in \bgamh$, every open subset
  $U \subset \projflags(d)$ containing $\phi^{-1}(p)$, every compact
  set $K \subset C_p$, and every finite subset $F \subset \Gamma_p$
  such that
  \[
    \rho(\Gamma_p \minus F)K \subset U,
  \]
  there is an open neighborhood $\mc{W}' \subset \mc{W}$ containing
  $\rho$, such that every $\rho' \in \mc{W}$ satisfies
  \[
    \rho'(\Gamma_p \minus F)K \subset U.
  \]
\end{definition}

\Cref{thm:stability_theorem} asserts that small deformations of EGF
representations inside of peripherally stable subspaces remain
EGF. The point of the peripheral stability condition is that it can be
verified by only considering how deformations of some representation
$\rho:\Gamma \to \PGL(d, \R)$ behave when they are restricted to
peripheral subgroups of $\Gamma$. If the peripheral subgroups are not
too complicated (for instance, if they are virtually nilpotent, or
even virtually abelian), then one can hope to get a direct
understanding of how their actions on $\projflags(d)$ deform in
certain subspaces of the representation variety
$\Hom(\Gamma, \PGL(d, \R))$; then \Cref{thm:stability_theorem} makes
it possible to upgrade this understanding to the deformed action of
the \emph{entire} group $\Gamma$ on the same subspace.

\subsection{EGF representations and relative Anosov representations}

As stated in the introduction, EGF representations generalize previous
definitions of relative Anosov representations provided by
Kapovich-Leeb \cite{kl2018relativizing}, Zhu \cite{zhu2019relatively},
and Zhu-Zimmer \cite{zz1}.

By a ``relative Anosov representation,'' in this paper we mean the
following:
\begin{definition}
  \label{defn:relative_anosov}
  Let $\Gamma$ be a subgroup of $\PGL(d, \R)$, and suppose that
  $(\Gamma, \mc{H})$ is a relatively hyperbolic pair. We say that
  $\Gamma$ is \emph{relatively $1$-Anosov} if $\Gamma$ is
  $1$-divergent, and there is a $\Gamma$-equivariant transverse
  embedding $\bgamh \to \projflags(d)$ whose image is the set of
  $(1,d-1)$-limit points of $\Gamma$.
\end{definition}

In the work of Kapovich-Leeb \cite{kl2018relativizing}, groups as in
\Cref{defn:relative_anosov} are called \emph{relatively asymptotically
  embedded} subgroups of $\PGL(d, \R)$; in \cite{zz1}, it is shown
that the definiton essentially agrees with Zhu's notion
\cite{zhu2019relatively} of a \emph{relatively dominated}
representation.

The relation to EGF representations is given by the following:
\begin{theorem}[{See \cite[Theorem 1.10]{Weisman2022}}]
  \label{thm:egf_relative_anosov}
  Let $\Gamma$ be a subgroup of $\PGL(d, \R)$, and suppose that
  $(\Gamma, \mc{H})$ is a relatively hyperbolic pair. Then the
  following are equivalent:
  \begin{enumerate}
  \item The group $\Gamma$ is relatively $1$-Anosov.
  \item The inclusion $\Gamma \hookrightarrow \PGL(d, \R)$ is a
    $1$-EGF representation, and there is an associated boundary
    extension $\phi:\Lambda \to \bgamh$ which is injective.
  \end{enumerate}
  Moreover, in this case the associated boundary set
  $\Lambda \subset \projflags(d)$ is precisely the set of
  $(1, d-1)$-limit points of $\Gamma$.
\end{theorem}


\section{Convex projective geometry: notation and background}
\label{sec:convex_projective_geometry}

In this section we briefly discuss some background material related to
convex projective geometry, as all of the examples of EGF
representations in this paper derive from convex projective structures
on manifolds. We also establish a few routine results that we will
need in later sections.

\begin{notation}
  We fix the rest of the following conventions for the rest of the
  paper.

  \begin{itemize}
  \item If $V$ is a real vector space, then $\P(V)$ denotes the
    projective space over $V$, i.e. the space $\Gr(1, V)$ of lines in
    $V$.

  \item For any $x \in V - \{0\}$, we let $[x]$ denote the image of
    $x$ under the quotient map $V - \{0\} \to \P(V)$.

  \item If $W$ is a subset of $V$, then we let $[W]$ denote the image
    of $W - \{0\}$ in $\P(V)$. If $W \subseteq V$ is a vector
    subspace, we will identify $\P(W)$ with $[W] \subset \P(V)$.

  \item When $B \subset \P(V)$, then the \emph{span} of $B$, denoted
    $\spn(B)$, is the subspace of $V$ spanned by any lift of $B$ in
    $V$. The \emph{projective span} of $B$ is the subset
    $[\spn(B)] \subset \P(V)$.
    
  \item If $w$ is any element of the \emph{dual} projective space
    $\P(V^*)$, we let $\hyp{w}$ denote the subset of $\P(V)$ given by
    the image of $\ker{\tilde{w}}$, where $\tilde{w}$ is any lift of
    $w$ in $V^*$.
    
  \item If $w \in \P(V^*)$ and $v \in \P(V)$ satisfy
    $v \notin \hyp{w}$, then we say $w$ and $v$ are \emph{transverse}
    and write $w \perp v$.
  \end{itemize}
\end{notation}

\subsection{Convex projective structures}

Let $V$ be a real $d$-dimensional vector space. Recall that a subset
$\Omega \subset \P(V)$ is \emph{properly convex} if
$\overline{\Omega}$ is a convex subset of an affine chart in
$\P(V)$. A manifold $M$ (possibly with boundary) has a \emph{convex
  projective structure} if $M$ can be realized as a quotient
$\Omega / \Gamma$ for a discrete subgroup
$\Gamma \subseteq \Aut(\Omega)$, where
\[
  \Aut(\Omega) = \{\gamma \in \PGL(V) : \gamma \cdot \Omega =
  \Omega\}.
\]

There are several different possible meanings for the ``boundary'' of
a properly convex set in $\P(V)$, so below we make things explicit:
\begin{definition}
  \label{defn:convex_set_boundary}
  Let $\Omega \subset \P(V)$ be a properly convex set with nonempty
  interior. In general, $\Omega$ might not be either open or closed.

  \begin{itemize}
  \item The \emph{frontier} of $\Omega$ is
    $\front(\Omega) = \overline{\Omega} - \mathrm{int}(\Omega)$.
  \item The \emph{nonideal boundary} of $\Omega$ is
    $\dee_n\Omega = \front(\Omega) \cap \Omega$.
  \item The \emph{ideal boundary} of $\Omega$ is
    $\dee_i\Omega = \front(\Omega) - \dee_n\Omega$.
  \end{itemize}

  When $\Omega$ is a properly convex \emph{open} set, we will use the
  notation $\dee \Omega$ to mean the \emph{ideal boundary}
  $\dee_i(\Omega)$, which coincides with the frontier $\front(\Omega)$
  in this case. Note that this conflicts with the convention in
  e.g. \cite{clt2018deforming,bcl2020generalized}.
\end{definition}

When $\Omega$ is open, we say it is a \emph{properly convex
  domain}. Then any manifold $M = \Omega / \Gamma$ has empty
boundary. If $M = \Omega / \Gamma$ is a convex projective manifold
with boundary, then its boundary $\dee M$ is identified with
$\dee_n\Omega / \Gamma$. When $\dee_n\Omega$ contains no nontrivial
projective segments, then we say that the manifold $M$ has
\emph{strictly convex boundary}.

\subsubsection{Faces in convex domains}

If $\Omega \subset \P(V)$ is a properly convex domain, a \emph{face}
of $\Omega$ is an equivalence class in $\dee \Omega$ under the
relation which identifies distinct points $x,y \in \dee \Omega$ if
there is an open projective line segment in $\dee \Omega$ containing
both $x$ and $y$.

If $x \in \dee \Omega$, we let $F_\Omega(x)$ denote the unique face of
$\Omega$ containing $x$. The \emph{support} of a face $F$, denoted
$\mathrm{supp}(F)$, is the projective span $[\spn(F)] \subset
\P(V)$. The \emph{boundary} of a face $\dee F$ is the boundary of $F$
when $F$ is viewed as a convex open subset of $\mathrm{supp}(F)$.

\subsubsection{The Hilbert metric}

Whenever $\Omega$ is a properly convex domain, one can define the
\emph{Hilbert metric} on $\Omega$ as follows:
\begin{definition}
  \label{defn:hilbert_metric}
  For a pair of distinct points $x, y \in \Omega$, we let
  \[
    d_\Omega(x, y) = \frac{1}{2}\log[u, v; x, y],
  \]
  where $u, v$ are the two points in $\dee \Omega$ such that
  $u, x, y, v$ lie on a projective line in that order, and
  $[a, b; c, d]$ is the \emph{cross-ratio}
  \[
    [a, b; c, d] = \frac{(d - a) (c - b)}{(c - a) (d - b)}.
  \]
  Here the differences can be measured under any affine identification
  of the projective line spanned by $x, y$ with $\R \cup \{\infty\}$.
\end{definition}
It turns out that $(\Omega, d_\Omega)$ is a proper geodesic metric
space, on which $\Aut(\Omega)$ acts by isometries. This means that
$\Aut(\Omega)$ acts properly on $\Omega$---in particular, discrete
subgroups of $\Aut(\Omega)$ act properly discontinuously.

The Hilbert metric gives us another perspective on the faces of
$\Omega$. Each face $F$ of $\Omega$ is a properly convex subset of
$\P(V)$, open in its own projective span. This allows us to define a
restricted Hilbert metric $d_F$ on $F$.

The proposition below is a standard result in the theory of convex
projective structures, and is a direct consequence of the definition
of the Hilbert metric.
\begin{prop}
  \label{prop:hilbert_metric_faces}
  Let $\Omega$ be a properly convex domain, let $F$ be a face of
  $\Omega$, and let $x_n$ be a sequence in $\Omega$ converging to some
  $x_\infty \in F$.

  For any $D > 0$, if $y_n \in \Omega$ is a sequence satisfying
  \[
    d_\Omega(x_n, y_n) \le D,
  \]
  then any accumulation point $y_\infty$ of $y_n$ lies in $F$, and
  \[
    d_F(x_\infty, y_\infty) \le D.
  \]
\end{prop}

\subsubsection{Dynamics of $\Aut(\Omega)$}
\label{sec:flag_dynamics}

Fix a divergent sequence $g_n \in \PGL(d, \R)$. This sequence is
$k$-divergent for some $k$ with $1 \le k < d$, so fix such a $k$, and
choose a subsequence of $g_n$ which is $k$-contracting and whose
sequence of inverses is $(d-k)$-contracting. There are then uniquely
determined projective subspaces $E_+, E_- \subset \R^d$ with
complementary dimension such that for any $x \in \P(\R^d) \minus E_-$,
the sequence $g_nx$ accumulates on $E_+$, uniformly on compact subsets
of $\P(\R^d) \minus E_-$. We refer to the subspaces $E_+$, $E_-$ as
\emph{attracting} and \emph{repelling} subspaces, respectively. We
emphasize that $E_+$ and $E_-$ are \emph{not} necessarily transverse.

The faces of $\Omega$ are related to attracting and repelling
subspaces of divergent sequences in $\Aut(\Omega)$:
\begin{prop}
  \label{prop:attracting_spaces_in_faces}
  Let $\gamma_n$ be a divergent sequence in $\Aut(\Omega)$ for a
  properly convex domain $\Omega$, and suppose that for some
  $x \in \Omega$, the sequence $\gamma_n x$ accumulates on a face
  $F_+$ of $\Omega$. Then, after extracting a subsequence, there is an
  attracting subspace $E_+$ of $\gamma_n$ such that
  $[E_+] \subseteq \supp(F_+)$.
\end{prop}
\begin{proof}
  Let $B$ be the ball of radius 1 about $x$ with respect to the
  Hilbert metric on $\Omega$, and let $x_\infty \in F_+$ be an
  accumulation point of the point of the sequence $\gamma_n x$. Using
  a diagonal argument, we can replace $\gamma_n$ with a subsequence so
  that for every $y$ in a countable dense subset of $B$, the sequence
  $\gamma_n \cdot y$ has a well-defined limit in the compact space
  $\P(V)$. \Cref{prop:hilbert_metric_faces} then implies that for
  every point $y \in B$, the sequence $\gamma_n y$ converges to a
  unique point in $F_+$.

  Let $B_\infty$ be the set of accumulation points of
  $\gamma_n \cdot B$, and let $W_\infty$ be the subspace
  $\spn(B_\infty) \subset V$.

  \Cref{prop:hilbert_metric_faces} implies that $B$ is a subset of the
  face $F = F_\Omega(x_\infty)$, so $[W_\infty]$ is a projective
  subspace of $\supp(F)$. Let $k = \dim(W_\infty)$. We claim that
  there is an open subset $U$ of the Grassmannian $\Gr(k, V)$ so that
  \[
    \gamma_nU \to \{W_\infty\}.
  \]
  This implies the desired result by
  e.g. \cite[Prop. 3.6]{Weisman2022}.

  To see the claim, fix $k$ points
  $z_1, \ldots, z_k \in B_\Omega(x, 1)$ so that the limits of the
  sequences $\gamma_n z_1, \ldots \gamma_n z_k$ span the projective
  subspace $[W_\infty]$. \Cref{prop:hilbert_metric_faces} implies that
  for some fixed $\eps > 0$, if $z_i'$ lies in the ball of radius
  $\eps$ about $z_i$, then the limits of the sequences
  \[
    \gamma_n z_1', \ldots, \gamma_nz_k'
  \]
  are in general position, and therefore also span $[W_\infty]$.

  For each $1 \le i \le k$, let $B_i$ denote the ball of radius $\eps$
  about $z_i$. Then, if $U$ is the open set
  \[
    \{W \in \Gr(k, V) : W = u_1 \oplus \ldots \oplus u_k, [u_i] \in
    B_i\},
  \]
  we have that $\gamma_n U \to \{W_\infty\}$, as required.
\end{proof}

\section{Convex cocompactness in projective space}
\label{sec:convex_cocompact}

Fix a real vector space $V$ of dimension $d$. Our goal in this section
is to prove \Cref{thm:convex_cocompact_egf}, which says that
\emph{convex cocompact} representations of relatively hyperbolic
groups in $\PGL(V)$ give examples of EGF representations. We begin by
recalling the precise definition of a (projectively) convex cocompact
group in $\P(V)$.

\begin{definition}[\cite{dgk2017convex}, Definitions 1.10 and 1.11]
  Let $\Omega$ be a properly convex domain in $\P(V)$, and let
  $\Gamma \subseteq \Aut(\Omega)$ be discrete.

  \begin{enumerate}
  \item The \emph{full orbital limit set} $\Lambda_\Omega(\Gamma)$ of
    $\Gamma$ is the union (over all $x \in \Omega$) of the set of
    accumulation points in $\dee \Omega$ of $\Gamma \cdot x$.
  \item The group $\Gamma$ \emph{acts convex cocompactly on $\Omega$}
    if the convex hull of $\Lambda_\Omega(\Gamma)$ is nonempty, and
    $\Gamma$ acts cocompactly on the convex hull of
    $\Lambda_\Omega(\Gamma)$ in $\Omega$.
  \item $\Gamma$ is \emph{convex cocompact in $\P(V)$} if it acts
    convex cocompactly on some properly convex domain
    $\Omega \subset \P(V)$.
  \end{enumerate}
\end{definition}

Danciger-Gu\'eritaud-Kassel show that if $\Gamma$ is a hyperbolic
group acting convex cocompactly in $\P(V)$, then the inclusion
$\Gamma \hookrightarrow \PGL(V)$ is a $1$-Anosov representation.

In \cite{weisman2020dynamical}, we showed that any group acting convex
cocompactly on some domain $\Omega \subset \P(V)$ acts with
``Anosov-like'' expansion dynamics on the faces of $\dee \Omega$. We
further showed that, when $(\Gamma, \mc{H})$ is a relatively
hyperbolic pair, $\Gamma$ acts convex cocompactly on $\Omega$, and
each $H \in \mc{H}$ also acts convex cocompactly on $\Omega$, then
there is an equivariant embedding of the Bowditch boundary $\bgamh$
into a \emph{quotient} of $\dee \Omega$.

In subsequent work \cite{iz2022structure}, Islam-Zimmer showed that
one does not need to assume that each $H \in \mc{H}$ acts convex
cocompactly on $\Omega$: it turns out that this follows automatically
from the fact that $\Gamma$ acts convex cocompactly. Together with our
earlier work, this gives the following:
\begin{theorem}[See \cite{weisman2020dynamical}, Theorem 1.16 and
  \cite{iz2022structure}, Theorem 1.6]
  \label{thm:convex_cocompact_boundary_embedding}
  Let $(\Gamma, \mc{H})$ be a relatively hyperbolic pair, and suppose
  $\rho(\Gamma)$ acts convex cocompactly on a properly convex domain
  $\Omega$. Then there is an equivariant homeomorphism
  \[
    \psi:\bgamh \to \Lambda_\Omega(\Gamma) / \sim,
  \]
  where $x \sim y$ if $x,y \in \Lambda_\Omega(H)$ for some
  $H \in \mc{H}$.
\end{theorem}

\begin{remark}
  Islam-Zimmer also proved a version of
  \Cref{thm:convex_cocompact_boundary_embedding} in the more general
  context of \emph{naive} convex cocompact group actions, which we do
  not discuss here.
\end{remark}

Before we can proceed with the proof of
\Cref{thm:convex_cocompact_egf}, we will need a slightly different
perspective on the full orbital limit set of a discrete group
$\Gamma \subseteq \Aut(\Omega)$.

\begin{definition}
  Let $\Omega$ be a convex projective domain. The \emph{dual domain}
  $\Omega^* \subset \P(V^*)$ is the set
  \[
    \Omega^* = \{w \in \P(V^*) : \hyp{w} \cap \overline{\Omega} =
    \emptyset\}.
  \]
\end{definition}

If $\Gamma \subset \PGL(V)$ is a subgroup of $\Aut(\Omega)$, then its
dual $\Gamma^* \subset \PGL(V^*)$ is a subgroup of
$\Aut(\Omega^*)$. If $\Gamma$ is convex cocompact in $\P(V)$, then
\cite{dgk2017convex}, Proposition 5.6 implies that there is some
$\Gamma$-invariant domain $\Omega$ so that $\Gamma$ acts convex
cocompactly on $\Omega$, and $\Gamma^*$ acts convex cocompactly on
$\Omega^*$.

\begin{definition}
  Let $\Gamma \subseteq \Aut(\Omega)$.

  \begin{enumerate}
  \item The \emph{dual full orbital limit set}
    $\Lambda_\Omega^*(\Gamma)$ is the full orbital limit set of
    $\Gamma^*$ in $\dee \Omega^*$.
  \item
    The \emph{flag-valued full orbital limit set}
    $\hat{\Lambda}_\Omega(\Gamma)$ is the set
    \[
      \hat{\Lambda}_\Omega(\Gamma) := \{(x, w) \in \projflags(V) : x
      \in \Lambda_\Omega(\Gamma), w \in \Lambda^*_\Omega(\Gamma)\}.
    \]
  \item The \emph{maximal domain} $\maxdomain(\Gamma)$ is the unique
    connected component of
    \[
      \P(V) - \bigcup_{w \in \Lambda_\Omega^*(\Gamma)} \hyp{w}
    \]
    containing $\Omega$. Equivalently, $\maxdomain(\Gamma)$ is the
    dual of the convex hull of $\Lambda_\Omega^*(\Gamma)$ in
    $\Omega^*$.
  \end{enumerate}
\end{definition}

We emphasize that $\maxdomain(\Gamma)$ is \emph{not} necessarily a
properly convex set, which means that we cannot always define a
Hilbert metric on it (so we do not have a guarantee that $\Gamma$ acts
properly discontinuously in general). However, when $\Gamma$ acts
convex cocompactly on $\Omega$, we do get a properly discontinuous
action, thanks to the following argument suggested by Jeff Danciger
and Fanny Kassel:

\begin{prop}
  \label{prop:maxdomain_accumulation}
  Let $\Gamma$ act convex cocompactly on a properly convex comain
  $\Omega$. For any sequence $\gamma_n \in \Gamma$ and any
  $x \in \maxdomain(\Gamma)$, the sequence $\gamma_n \cdot x$
  accumulates in $\Lambda_\Omega(\Gamma)$, uniformly on compacts. In
  particular, $\Gamma$ acts properly discontinuously on
  $\maxdomain(\Gamma)$.
\end{prop}
\begin{proof}
  When $\maxdomain(\Gamma)$ is a \emph{properly} convex domain, this
  follows immediately from Proposition 4.18 in \cite{dgk2017convex},
  which says that whenever $\Gamma$ acts convex cocompactly on some
  domain $\Omega$, and $\Omega'$ is any $\Gamma$-invariant properly
  convex domain containing $\Omega$, then $\Gamma$ acts convex
  cocompactly on $\Omega'$ and
  $\Lambda_\Omega(\Gamma) = \Lambda_{\Omega'}(\Gamma)$.

  So, we consider the case where $\maxdomain(\Gamma)$ is \emph{not}
  properly convex. We may assume our domain $\Omega$ is chosen so that
  $\Gamma$ acts convex cocompactly on both $\Omega$ and
  $\Omega^* \subset \P(V^*)$. Since $\maxdomain(\Gamma)$ is not
  properly convex, its dual $\maxdomain(\Gamma)^*$ (given by the
  convex hull of $\Lambda^*_\Omega(\Gamma)$ in $\Omega^*$) has empty
  interior (i.e. it spans a proper projective subspace of $\P(V^*)$).

  Given any $\eps > 0$, we let $\Omega_\eps^*$ be the uniform
  $\eps$-neighborhood of $\maxdomain(\Gamma)^*$, with respect to the
  Hilbert metric on $\Omega^*$. We let $\Omega_\eps \subset \P(V)$
  denote the dual of $\Omega_\eps^*$. Note that $\Omega_\eps$ is a
  $\Gamma$-invariant properly convex subset of $\maxdomain(\Gamma)$,
  containing $\Omega$.

  Since duality reverses inclusions, and the intersection
  \[
    \bigcap_{\eps \to 0} \Omega_\eps^*
  \]
  is exactly the set $\maxdomain(\Gamma)^*$, the union
  \[
    \bigcup_{\eps \to 0} \Omega_\eps
  \]
  is the set $\maxdomain(\Gamma)$. So, if we fix a compact set
  $K \subset \maxdomain(\Gamma)$, for some $\eps > 0$ we have
  $K \subset \Omega_\eps$. Then we apply Proposition 4.18 in
  \cite{dgk2017convex} to the \emph{properly} convex domain
  $\Omega_\eps$ to see that for any sequence $\gamma_n \in \Gamma$,
  $\gamma_n \cdot K$ accumulates in $\Lambda_\Omega(\Gamma)$.
\end{proof}

We need a few more results before we can prove
\Cref{thm:convex_cocompact_egf}. We quote the following observation
from \cite{weisman2020dynamical}:
\begin{prop}[\cite{weisman2020dynamical}, Proposition 8.1]
  \label{prop:segments_in_peripherals}
  Let $\Gamma \subset \PGL(V)$ act convex cocompactly on a properly
  convex domain $\Omega$. If $\Gamma$ is hyperbolic relative to a
  collection of subgroups $\mc{H}$ also acting convex cocompactly on
  $\Omega$, then every nontrivial projective segment in
  $\Lambda_\Omega(\Gamma)$ is contained in $\Lambda_\Omega(H)$ for
  some $H \in \mc{H}$.
\end{prop}

Using a theorem of Danciger-Gu\'eritaud-Kassel (\cite{dgk2017convex},
Theorem 1.18), one can strengthen this result:
\begin{cor}
  \label{cor:segments_in_peripherals}
  In the context of \Cref{prop:segments_in_peripherals}, it is
  possible to choose the convex domain $\Omega$ so that every
  nontrivial projective segment in $\dee \Omega$ is contained in
  $\Lambda_\Omega(H)$ for some $H \in \mc{H}$.
\end{cor}

We also observe:
\begin{prop}
  \label{prop:boundary_embedding_conical_pts}
  Let $(\Gamma, \mc{H})$ be a relatively hyperbolic pair, let
  $\rho:\Gamma \to \PGL(V)$ be representation such that $\rho(\Gamma)$
  acts convex cocompactly on a domain $\Omega$, and let
  $\psi:\bgamh \to \Lambda_\Omega(\Gamma) / \sim$ be the map
  coming from \Cref{thm:convex_cocompact_boundary_embedding}.

  If $z \in \bgamh$ is a conical limit point, and $\gamma_n$ is a
  sequence limiting to $z$ in
  $\overline{\Gamma} = \Gamma \sqcup \bgamh$, then $\psi(z) \in \P(V)$
  is the unique one-dimensional attracting subspace for
  $\rho(\gamma_n)$ in $\P(V)$.
\end{prop}
\begin{proof}
  It suffices to show that any subsequence of $\gamma_n$ has a further
  subsequence which has $\psi(z)$ as a one-dimensional attracting
  subspace. So, using the convergence group property, we can take a
  subsequence and assume that there is some point $y \in \bgamh$ so
  that $\gamma_n$ converges to $z$ on every point in the set
  $\bgamh - \{y\}$. We can further assume that the pair
  $(\Gamma, \mc{H})$ is not elementary, so $\bgamh$ contains
  infinitely many points.

  So, by \Cref{cor:segments_in_peripherals}, we can find distinct
  points $u, v \in \Lambda_\Omega(\Gamma)$ so that the projective line
  segment $(u,v)$ lies in $\Omega$, and $\rho(\gamma_n)u$,
  $\rho(\gamma_n)v$ both converge to $\psi(z)$. Corollary 4.10 in
  \cite{dgk2017convex} implies that
  $F_\Omega(\psi(z)) \subset \Lambda_\Omega(\Gamma)$, and then
  \Cref{cor:segments_in_peripherals} implies that
  $F_\Omega(\psi(z)) = \{\psi(z)\}$.

  Then for any $x \in (u,v)$, $\rho(\gamma_n)x$ converges to
  $\psi(z)$, and we are done by
  \Cref{prop:attracting_spaces_in_faces}.
\end{proof}

\begin{proof}[Proof of \Cref{thm:convex_cocompact_egf}]
  Fix a $d$-dimensional real vector space $V$, let $(\Gamma, \mc{H})$
  be a relatively hyperbolic pair, and let $\rho:\Gamma \to \PGL(V)$
  be a representation such that $\rho(\Gamma)$ acts convex cocompactly
  on a properly convex domain $\Omega \subset \P(V)$. This implies
  (see \cite{iz2022structure}) that each $H \in \mc{H}$ also acts
  convex cocompactly on $\Omega$.

  The first step in the proof is to define a $\rho(\Gamma)$-invariant
  subset $\hat{\Lambda} \subset \projflags(V)$ and a boundary
  extension $\hat{\phi}:\hat{\Lambda} \to \bgamh$. We use the map
  $\psi$ coming from \Cref{thm:convex_cocompact_boundary_embedding} to
  define an equivariant surjection $\phi:\Lambda \to \bgamh$, where
  the preimage of each parabolic point $p$ in $\bgamh$ is exactly
  $\Lambda_\Omega(H)$ for $H = \mr{Stab}_\Gamma(p)$.

  If we let $\Lambda^*$ be the full orbital limit set of the
  $\Gamma$-action on $\Omega^*$, we can similarly find an equivariant
  surjection $\phi^*:\Lambda^* \to \bgamh$, where
  ${\phi^*}^{-1}(z)$ is a single hyperplane if $z$ is a conical limit
  point, and the dual full orbital limit set of $\mr{Stab}_\Gamma(z)$
  if $z$ is a parabolic point.

  We consider the set
  \[
    \hat{\Lambda} = \{ (x, w) \in \projflags : x \in \Lambda, w \in
    \Lambda^*\}.
  \]

  Each element of $\Lambda^*$ is a supporting hyperplane of the domain
  $\Omega$. \Cref{cor:segments_in_peripherals} implies that for every
  point $(x, w)$ in $\hat{\Lambda}$, either:
  \begin{enumerate}
  \item $x = \phi^{-1}(z)$ and $w = {\phi^*}^{-1}(z)$ for a conical
    limit point $z \in \bgamh$, or
  \item $x \in \Lambda_H$ and $w \in \Lambda_H^*$ for a peripheral
    subgroup $H \in \mc{H}$.
  \end{enumerate}
  This allows us to combine $\phi$ and $\phi^*$ to get a well-defined
  equivariant surjection $\hat{\phi}:\hat{\Lambda} \to \bgamh$.

  The next step is to define the open subsets $C_z \subset \projflags$
  for each $z \in \bgamh$. If $z \in \bgamh$ is a conical
  limit point, we define the set $C_z$ by
  \[
    C_z = \{\nu \in \projflags : \nu \textrm{ is opposite to }
    \phi^{-1}(z)\}.
  \]

  Otherwise, if $z$ is a parabolic point, we consider the maximal
  domain $\maxdomain(H) \subset \P(V)$ for $H =
  \Stab_\Gamma(z)$. Dually, we can define
  $\Omega^*_{\mathrm{max}}(H) \subset \P(V^*)$, viewing
  $\Lambda_\Omega(H)$ as the dual full orbital limit set of $H^*$
  acting on $\Omega^*$.

  Then, we define
  \[
    C_z = \{(x, w) \in \projflags(V) : x \in \maxdomain(H), w \in
    \Omega^*_{\mathrm{max}}(H)\}.
  \]
  
  For every $z \in \bgamh$, $C_z$ is open, and
  \Cref{cor:segments_in_peripherals} implies that $C_z$ contains
  $\hat{\phi}^{-1}(z')$ for every $z' \ne z$ in $\dee
  \Gamma$.

  The last step is to check that the map $\hat{\phi}$ actually extends
  the convergence group action of $\Gamma$ on $\bgamh$, using the sets
  $C_z$. To do so, we appeal to
  \Cref{prop:conical_peripheral_implies_egf}.

  First, let $\gamma_n$ be a sequence in $\Gamma$ limiting conically
  to $z \in \bgamh$. \Cref{prop:boundary_embedding_conical_pts}
  implies that $\phi^{-1}(z)$ is the unique one-dimensional attracting
  subspace for $\rho(\gamma_n)$ in $\P(V)$. Dually, $(\phi^*)^{-1}(z)$
  is the unique one-dimensional attracting subspace for
  $\rho(\gamma_n)$ in $\P(V^*)$. So the sequence $\rho(\gamma_n)$ is
  $(1, d-1)$-divergent, and so is the sequence $\rho(\gamma_n^{-1})$.

  Further, observe that if $\gamma_n$ is any sequence in $\Gamma$ such
  that $\rho(\gamma_n)$ is $1$-divergent, then any $1$-limit point of
  $\rho(\gamma_n)$ must lie $\Lambda_\Omega(\Gamma)$. The same holds
  for $(d-1)$-divergent sequences and
  $\Lambda^*_\Omega(\Gamma)$. Together this implies that the first
  condition of \Cref{prop:conical_peripheral_implies_egf} is
  satisfied.

  On the other hand, if $z$ is a parabolic point and $\gamma_n$ is an
  infinite sequence in $\Stab_\Gamma(z)$,
  \Cref{prop:maxdomain_accumulation} implies that for any compact
  $K \subset C_z$, the set $\rho(\gamma_n) \cdot K$ eventually lies in
  any given neighborhood of $\phi^{-1}(z)$, which fulfills the second
  condition of \Cref{prop:conical_peripheral_implies_egf}.
\end{proof}


\section{Generalized cusps}
\label{sec:generalized_cusps}

In this section, we wish to consider deformations of the holonomy of a
finite-volume convex projective manifold $M = \Omega / \Gamma$, where
$\Omega$ is a strictly convex subset of projective space with nonempty
interior. (Note that, in contrast to the previous section, in this
section we assume that $\Omega$ is strictly convex, but we do
\emph{not} assume that it is an open subset of projective space).

As mentioned in the introduction to this paper, whenever $M$ is such a
manifold, work of Cooper-Long-Tillman \cite{clt2015convex},
Crampon-Marquis \cite{cm2014finitude}, and Zhu
\cite{zhu2019relatively} implies that the holonomy representation
$\rho:\pi_1M \to \PGL(d, \R)$ is relatively $1$-Anosov in the sense of
\Cref{defn:relative_anosov}, and therefore an EGF representation by
\Cref{thm:egf_relative_anosov}.

Our goal in this section is to prove \Cref{thm:vf_deformations}, which
says that certain deformations of $\rho$ are peripherally stable, and
hence give rise to EGF representations. These peripherally stable
deformation spaces can contain representations $\rho'$ preserving a
properly convex domain $\Omega'$ which is \emph{not} strictly
convex. The quotient $\Omega'/\rho'(\pi_1M)$ is in general an
\emph{infinite}-volume convex projective manifold homeomorphic to $M$:
it is the union of a compact piece and finitely many \emph{generalized
  cusps}. We provide the definition below:

\begin{definition}[See \cite{clt2018deforming},
  \cite{bcl2020generalized}]
  Let $\Omega \subset \P(\R^d)$ be a properly convex set with nonempty
  interior, and let $\Gamma \subseteq \Aut(\Omega)$ be discrete. A
  manifold $C = \Omega / \Gamma$ is a \emph{generalized cusp} if $C$
  has compact and strictly convex boundary, $\Gamma \simeq \pi_1C$ is
  virtually abelian, and $C$ is homeomorphic to
  $\dee C \times [0, \infty)$.
\end{definition}

\subsection{Flag stability for generalized cusps}

In \cite{clt2018deforming}, Cooper-Long-Tillmann prove that if one
deforms the holonomy representation of a generalized cusp in certain
controlled way, then (for a small enough deformation), the resulting
representation is still the holonomy of a generalized cusp. To state
their result, for any virtually abelian group $H$, let
$\Homvf(H, \PGL(d, \R)$ denote the space of \emph{virtual
  flag} representations $\rho:\Gamma \to \PGL(d, \R)$, consisting of
representations whose image is virtually conjugate to a group of
upper-triangular matrices.

\begin{theorem}[\cite{clt2018deforming}, Theorem 6.25]
  \label{thm:cusps_open_in_vfg}
  Let $C$ be a generalized cusp, and let $\mc{U}$ be the set of
  holonomies of convex projective structures on $C$ with strictly
  convex boundary. Then $\mc{U}$ is an open subset of
  $\Homvf(\pi_1C)$.
\end{theorem}

Using a gluing procedure, Cooper-Long-Tillmann can then show that
whenever $M$ is a finite-volume convex projective manifold, certain
restricted deformations of its holonomy give rise to new convex
projective structures on $M$. (In fact, their result holds in greater
generality, but in this paper we consider only the finite-volume
case.) For this result, recall from the introduction that when
$(\Gamma, \mc{H})$ is a relatively hyperbolic pair,
$\Homvf(\Gamma, \PGL(d, \R), \mc{H})$ denotes the space of
representations $\rho:\Gamma \to \PGL(d, \R)$ such that the
restriction of $\rho$ to each $H \in \mc{H}$ lies in
$\Homvf(H, \PGL(d, \R))$.

\begin{theorem}[See \cite{clt2018deforming}, Theorem 0.1]
  \label{thm:clt_gluing}
  Let $M = \Omega/\Gamma$ be a finite-volume convex projective
  manifold with holonomy $\rho:\Gamma \to \PGL(d, \R)$. Let $\mc{H}$
  be the collection of conjugates of cusp groups of $M$. Then an open
  neighborhood of $\rho$ in
  $\Homvf(\Gamma, \PGL(d, \R), \mc{H})$ consists of
  holonomy representations of convex projective manifolds homeomorphic
  to $M$, each of which is the union of a compact piece and finitely
  many generalized cusps.
\end{theorem}

We can think of the relative stability theorem for EGF representations
as a kind of dynamical analog of this ``geometric gluing'' result,
since \Cref{thm:stability_theorem} gives us a way to understand the
\emph{topological dynamics} of the deformation of some representation,
provided we have control over the topological dynamics of the
restrictions of the representation to its cusp
groups. \Cref{thm:vf_deformations} says that small deformations inside
of $\Homvf(\Gamma, \PGL(d, \R), \mc{H})$ are controlled in
precisely the sense we need for this to apply.

Proving \Cref{thm:vf_deformations} therefore boils down to analyzing
the dynamics implicit in \Cref{thm:cusps_open_in_vfg}: using the
geometry of generalized cusps, we show that if
$\rho:\pi_1C \to \PGL(d, \R)$ is the holonomy of a hyperbolic cusp,
then the large-scale dynamics of any nearby representation
$\rho':\pi_1C \to \PGL(d, \R)$ in
$\Homvf(\pi_1C, \PGL(d, \R))$ are ``close'' to the
dynamics of $\rho$ (in the sense made precise by the definition of
peripheral stability).

\begin{remark}
  We expect that a version of \Cref{thm:vf_deformations} also holds
  with weaker assumptions on $M$, which are more in line with the
  original assumptions in the Cooper-Long-Tillmann stability
  result. For instance, we conjecture that whenever $M$ is a convex
  projective manifold with strictly convex boundary, each end of $M$
  is a generalized cusp, and $\pi_1M$ is relatively hyperbolic
  (relative to a collection of subgroups $\mc{H}$), then the holonomy
  of $M$ is an EGF representation, and
  $\Homvf(\pi_1M, \PGL(d, \R), \mc{H})$ is a peripherally
  stable subspace.
\end{remark}

\subsection{Generalized horospheres}

In \cite{clt2018deforming}, Cooper-Long-Tillmann show that if $C$ is a
generalized cusp with holonomy $\rho:\pi_1C \to \PGL(d,\R)$, there is
a finite-index subgroup $\Gamma_1 \subseteq \pi_1C$ (depending only on
$\pi_1C$ and $d$) so that $\rho(\Gamma_1)$ is a lattice in a
\emph{syndetic hull} of $\rho(\Gamma_1)$: a uniquely determined
connected Lie group $T(\rho) \subset \PGL(d, \R)$, conjugate into the
group of upper triangular matrices. This group is called the
\emph{translation group} of the cusp.

We may assume that $\Gamma_1$ is free abelian, so it is a lattice in
$\Gamma_1 \ten_\Z \R \simeq \R^{d-2}$. The restriction
\[
  \rho|_{\Gamma_1}:\Gamma_1 \to \PGL(d, \R)
\]
extends uniquely to an embedding of Lie groups
\[
  \iota_\rho:\R^{d-2} \hookrightarrow \PGL(d, \R)
\]
with image $T(\rho)$.

We observe the following:
\begin{prop}
  \label{prop:cusp_embedding_c1_continuous}
  Let $C$ be a generalized cusp. The embedding
  \[
    \iota_\rho:\R^{d-2} \to \PGL(d, \R)
  \]
  varies continuously with
  $\rho \in \Homvf(\pi_1C, \PGL(d, \R))$ in the
  compact-open topology on maps from $\R^{d-2}$ into $\PGL(d, \R)$.
\end{prop}
\begin{proof}
  The Lie algebra of $\PGL(d, \R)$ is identified with
  $\mf{sl}(d, \R)$. Fix a finite set $S$ of generators for
  $\Gamma_1$. The subspace $\mf{s}_\rho \subset \mf{sl}(d, \R)$
  spanned by $\log(\rho(S))$ varies smoothly with $\rho$, and the
  induced map
  \[
    \R^{d-2} \to \mf{sl}(d, \R)
  \]
  with image $\mf{s}_\rho$ varies continuously in the compact-open
  topology. The embedding $\iota_\rho$ is given by composition with
  the exponential map.
\end{proof}

When $C = \Omega / \Gamma$ is projectively equivalent to a cusp in
some hyperbolic manifold, we can assume that $\Omega$ is a closed
horoball in $\H^d$. The nonideal boundary $\dee_n\Omega$ is a
horosphere preserved by $\Gamma$, which carries a $\Gamma$-invariant
Euclidean metric. In this case the translation group is the group of
Euclidean translations on $\dee_n\Omega$.

When $C$ is a generalized cusp, we can always find some orbit of the
translation group $T(\rho)$ in $\P(\R^d)$ which is a strictly convex
hypersurface (see Proposition 6.22 in \cite{clt2018deforming}). This
hypersurface is called a \emph{generalized horosphere}. Its convex
hull in $\P(\R^d)$ is a \emph{generalized horoball}. We can always
choose this horoball so that its quotient by $\Gamma$ is contained in
the generalized cusp $C$.

Now consider a convex projective manifold $M = \Omega / \Gamma$ with
strictly convex boundary which can be written as a union of a compact
piece and finitely many generalized cusps. It is always possible to
choose the cusps so that each cusp $C \subset M$ is a quotient
$C = \Omega_C / \Gamma_C$, where $\Omega_C \subset \Omega$ is a convex
subset whose nonideal boundary is a generalized horosphere, and
$\Gamma_C = \pi_1C$ is the cusp group. The boundary $\dee_n\Omega_C$
is homogeneous, in the sense that the translation group $T(\Gamma_C)$
acts simply transitively on it.

\subsection{The ideal boundary}

The Ballas-Cooper-Leitner classification of generalized cusps allows
us to get a more explicit description of the ideal boundary of
$\Omega$. Given a generalized cusp $C = \Omega / \Gamma$, we let
$\Omega_C$ denote the ``standard'' $\Gamma$-invariant domain with
homogeneous nonideal boundary, alluded to above.

\begin{prop}[Lemmas 1.24 and 1.25 in \cite{bcl2020generalized}]
  Let $C = \Omega_C / \Gamma$ be a generalized cusp. The ideal
  boundary of $\Omega_C$ is a projective $k$-simplex $\Delta_C$. There
  is a unique supporting hyperplane $H_C$ of $\Omega_C$ containing
  $\Delta_C$, and the affine chart
  \[
    \A_C = \P(\R^d) - H_C
  \]
  is the unique affine chart containing $\Omega_C$ as a closed subset.
\end{prop}

The vertices of $\Delta_C$ must be preserved by $\Gamma$, and in fact
they are all eigenvectors for the translation group $T(\Gamma)$.

Each generalized horosphere $S_C$ for $C$ is a strictly convex
hypersurface contained in the affine chart $\A_C$. The closure of this
hypersurface in $\P(\R^d)$ is either $S_C \cup \dee\Delta_C$ (if $C$
is a ``type $d-1$'' cusp) or $S_C \cup \Delta_C$ (if $C$ is any other
type of cusp).

\subsection{Deformations of convex hypersurfaces}

The main ingredient in the proof of \Cref{thm:vf_deformations} is the
following:

\begin{lem}
  \label{lem:convex_hypersurface_deformation}
  Let $C$ be a hyperbolic cusp with holonomy $\rho$. Let $p_C$ be the
  cusp point, and let $H_C$ be the unique supporting hyperplane of
  $\Omega_C$ at $p_C$.

  Let $K$ be a compact subset in $\A_C = \P(\R^d) - H_C$, let
  $U \subset \P(\R^d)$ be an open subset containing $p_C$, and let
  $F \subset \pi_1C$ be a cofinite subset such that
  $\rho(F)K \subset U$.

  Then there exists a neighborhood $\mc{W}$ of $\rho$ in
  $\Homvf(\pi_1C, \PGL(d, \R))$ so that for any $\rho' \in \mc{W}$, we
  have
  \[
    \rho'(F)K \subset U.
  \]
\end{lem}
\begin{proof}
  \cite{clt2018deforming}, Theorem 6.25 implies that we can choose a
  neighborhood $\mc{W}$ of $\rho$ in
  $\Homvf(\pi_1C, \PGL(d, \R))$ consisting of holonomies
  of generalized cusps. For any $\rho' \in \mc{W}$, we let $\Omega'$
  denote a ``standard'' properly convex set preserved by
  $\rho'(\pi_1C)$, whose non-ideal boundary is a generalized
  horosphere.

  Since $p_C$ and $H_C$ are respectively the unique eigenvector and
  fixed hyperplane of $\rho(\pi_1C)$, we can choose our neighborhood
  $\mc{W}$ so that for any $\rho' \in \mc{W}$, any eigenvectors and
  fixed hyperplanes of $\rho'(\pi_1C)$ are close to $p_C$, $H_C$.

  In particular, we can choose $\mc{W}$ so that the ideal boundary of
  $\Omega'$ is a $k$-simplex $\Delta'$ contained in $U$. And, by
  applying a small conjugation in $\PGL(d+1, \R)$, we can assume that
  $H_C$ is the unique supporting hyperplane of $\Omega'$ containing
  $\Delta'$.

  Let $T(\rho)$ be the translation group of $\rho$. For any $x \in K$,
  the orbit $T(\rho) \cdot x$ is a paraboloid in $\A_C$, and as $x$
  varies in $K$ we obtain a family of paraboloids foliating a region
  of $\A_C$. We can write $\A_C = \R^{d-1} \times \R$, and then view
  each paraboloid as the graph of a function
  \[
    f_\rho:\R^{d-2} \to \R.
  \]
  The function $f_\rho$ is determined by the condition
  \[
    (u, f_\rho(u)) = \iota_\rho(u) \cdot x.
  \]
  Here $\iota_\rho:\R^{d-2} \to \PGL(d, \R)$ restricts to $\rho$ on a
  finite-index subgroup $\Gamma_1 \subset \pi_1C$, with $\Gamma_1$
  identified with $\Z^{d-2} \subset \R^{d-2}$.

  If $T(\rho')$ is the translation group of $\rho'$, then Lemma 6.24
  in \cite{clt2018deforming} implies that for each $x \in K$, the
  orbit $T(\rho') \cdot x$ is a strictly convex hypersurface
  $S' \subset \A_C$. The hypersurface $S'$ is the graph of a map
  $f_{\rho'}:\R^{d-2} \to \R$, satisfying
  \[
    (u, f_{\rho'}(u)) = \iota_{\rho'}(u) \cdot x.
  \]
  \Cref{prop:cusp_embedding_c1_continuous} implies that $f_{\rho'}$
  varies continuously (in the compact-open topology on continuous maps
  $\R^{d-2} \to \R$) as $\rho'$ varies in $\mc{W}$ and as $x$ varies
  in $K$.

  We fix a norm $||\cdot||$ on $\R^{d-2}$. There is a constant $D$ so
  that for any $(u, v) \in \A_C$, if $||u|| > 1$ and
  $|v| / ||u|| > D$, then $(u, v)$ is contained in the neighborhood
  $U$ of $p_C$.

  The function $f_{\rho'}$ is strictly convex, and we can assume that
  it is nonnegative and uniquely minimized at the origin. So, if
  $f_{\rho'}(u) / ||u|| > D$ on $\{u \in \R^{d-2} : ||u|| = M\}$ for
  some constant $M$, then $f_{\rho'}(u) / ||u|| > D$ for all $u$ with
  $||u|| \ge M$.

  So, as long as $\mc{W}$ is sufficiently small, there is a fixed ball
  $B \subset \R^{d-1}$ so that if $u \in \R^{d-1} - B$, then
  \[
    (u, f_{\rho'}(u)) \in U
  \]
  for any $\rho' \in \mc{W}$. As both the constants $D, M$ above vary
  continuously as $x$ varies in the compact subset $K$, we can choose
  $\mc{W}$ so that the above holds for any $f_{\rho'}$ determined by
  any $x \in K$.

  The ball $B$ contains at most finitely many elements of
  $\Gamma_1 - F$. So we can in fact choose $\mc{W}$ small enough so
  that for any $\rho' \in \mc{W}$, the set $\rho'(\Gamma_1 - F)K$ lies
  in $U$. Then since $\Gamma_1$ has finite index in $\pi_1C$, we can
  shrink $\mc{W}$ even further to guarantee that for any
  $\rho' \in \mc{W}$,
  \[
    \rho'(\pi_1C - F)K
  \]
  lies in $U$ as well.
\end{proof}

\subsection{Peripheral stability}

\begin{proof}[Proof of \Cref{thm:vf_deformations}]
  Let $\rho:\pi_1M \to \PGL(d, \R)$ be the holonomy of a finite-volume
  convex projective manifold $M$, and let $\Omega$ be a
  $\rho$-invariant strictly convex domain such that
  $M = \Omega / \rho(\pi_1M)$. Write $\Gamma = \pi_1M$, and let
  $\mc{H}$ be the collection of cusp groups, so $(\Gamma, \mc{H})$ is
  a relatively hyperbolic pair, and for each $H \in \mc{H}$, $\rho|_H$
  is the holonomy of a hyperbolic cusp.

  Since $\Omega$ is strictly convex, \cite[Theorem
  11.6]{clt2015convex} also implies that $\Omega$ has $C^1$
  boundary. So there is a $\rho$-equivariant homeomorphism
  $\dee \Omega \to \dee \Omega^*$ assigning the point
  $z \in \dee \Omega$ to the unique supporting hyperplane of $\Omega$
  at $z$. We let $\dee \hat{\Omega}$ denote the space
  \[
    \{(x, w) \in \projflags(d) : x \in \dee \Omega, w \in \dee
    \Omega^*\}.
  \]
  
  There is an equivariant homeomorphism
  $\psi:\bgamh \to \dee \hat{\Omega}$, with the parabolic points in
  $\bgamh$ corresponding to the fixed flags of the cusp groups. The
  inverse map $\phi:\dee\hat{\Omega} \to \bgamh$ extends the
  convergence action of $\Gamma$ on $\bgamh$. For each parabolic point
  $p \in \bgamh$, the open set $C_p$ is
  \[
    \Opp(\psi(p)) = \{\xi \in \projflags : \xi \textrm{ is opposite to
    } \psi(p)\}.
  \]

  Let $\pi:\projflags(d) \to \P(\R^d)$ be the canonical projection
  map. It suffices to show that for any compact set
  $K \subset \pi(C_p)$, any open neighborhood $U$ of $\pi(\psi(p))$ in
  $\P(\R^d)$, and any cofinite subset
  $F \subset \Gamma_p = \Stab_\Gamma(p)$ such that
  \[
    \rho(F) \cdot K \subset U,
  \]
  we can find a neighborhood
  $\mc{W} \subset \Homvf(\Gamma_p, \PGL(d, \R), \mc{H})$ containing
  $\rho$ such that
  \[
    \rho'(F) \cdot K \subset U
  \]
  for any $\rho' \in \mc{W}$. However, this is exactly the content of
  \Cref{lem:convex_hypersurface_deformation}. The same argument
  applied dually then shows that we can upgrade $K$ to a compact
  subset of $C_p \subset \projflags$ and $U$ to an open subset in
  $\projflags(d)$, giving the required peripheral stability.
\end{proof}


\section{Symmetric representations}
\label{sec:symmetric_powers.tex}

\label{sec:symmetric_powers}

In this section, we construct new examples of extended geometrically
finite representations by taking symmetric powers of convex cocompact
representations of groups which are hyperbolic relative to virtually
abelian subgroups. We also prove \Cref{thm:symmetric_powers_stable},
which states that these EGF representations are \emph{absolutely}
stable in the representation variety.

\subsection{Symmetric powers}

Let $V$ be a finite-dimensional real vector space. We let $\tau_m$
denote the symmetric representation
\[
  \SL(V) \to \SL(\Sym^m(V)).
\]
Throughout this section, we will view $\Sym^m(V)$ as a quotient of the
space of homogeneous degree-$m$ polynomials in elements of $V$. We
will always leave this quotient implicit. That is, if
$v_1, \ldots, v_k \in V$, and $r_1, \ldots, r_k \in \N \cup \{0\}$
with $\sum r_i = m$, we will view the monomial
$v_1^{r_1}\cdots v_k^{r_k}$ as an element of $\Sym^m(V)$.

There is a $\tau_m$-equivariant embedding
\[
  \iota:\P(V) \to \P(\Sym^m(V))
\]
given by $[v] \mapsto [v^m]$. There is also a corresponding dual
embedding
\[
  \iota^*:\P(V^*) \to \P(\Sym^m(V)^*),
\]
using the canonical identification $\Sym^m(V^*) \simeq \Sym^m(V)^*$.
We observe that $v \in \P(V)$ and $w \in \P(V^*)$ are transverse if
and only if their respective images under $\iota$ and $\iota^*$ are
also transverse. This means that the maps $\iota$ and $\iota^*$ also
give rise to a $\tau_m$-equivariant map
\[
  \hat{\iota}:\projflags(V) \to \projflags(\Sym^m(V))
\]
given by $\hat{\iota}(v, w) = (\iota(v), \iota^*(w))$.

\subsubsection{Dynamics in symmetric powers}

The dynamics of $1$-divergent sequences in $\SL(V)$ on $\P(V)$ and
$\P(V^*)$ are respected in $\P(\Sym^mV)$, in sense given by the
following (easily verified) proposition:
\begin{prop}
  \label{prop:symmetric_power_dynamics}
  Let $\{g_n\}$ be an infinite sequence in $\SL(V)$ such that for some
  $w \in \P(V^*)$, $x \in \P(V)$, we have $g_n|_{\Opp(w)} \to x$
  uniformly on compacts.
  
  Then
  \[
    \tau_m(g_n)|_{\Opp(\iota^*(w))} \to \iota(x)
  \]
  uniformly on compacts.
\end{prop}

\subsection{Symmetric powers of relatively hyperbolic convex cocompact
  groups}

Suppose that $\Gamma$ is a convex cocompact subgroup of $\PGL(V)$, so
that there is a properly convex domain $\Omega \subset \P(V)$ with
$\Gamma$ acting convex cocompactly on $\Omega$. We may replace
$\Gamma$ with a finite index-subgroup and lift $\Gamma$ to a
representation $\rho:\Gamma \to \SL(V)$. In this situation we say that
the representation $\rho:\Gamma \to \SL(V)$ is \emph{convex cocompact
  in $\P(V)$}.

Let $(\Gamma, \mc{H})$ be a relatively hyperbolic pair such that each
group in $\mc{H}$ is virtually abelian, and let
$\rho:\Gamma \to \SL(V)$ be a convex cocompact
representation. Representations of this form have been studied
extensively by Islam-Zimmer \cite{iz2019convex}, \cite{iz2019flat},
who proved a number of strong structural results. In particular,
Islam-Zimmer showed that in this situation, for each $H \in \mc{H}$,
the image $\rho(H)$ acts cocompactly on a properly embedded
$k$-simplex $\Delta_H \subset \Omega$, with $k = \mathrm{rank}(H)$. (A
simplex $\Delta \subset \Omega$ is \emph{properly embedded} if
$\dee \Delta \subset \dee \Omega$.) In fact, the action of $\rho(H)$
on $\Omega$ is convex cocompact, and the full orbital limit set of
$\rho(H)$ in $\Omega$ is $\dee \Delta_H$.

Conversely, every properly embedded maximal $k$-simplex in the convex
hull of $\Lambda_\Omega(\Gamma)$ always has a cocompact action by some
$H \in \mc{H}$ with rank $k$.

We let
\[
  \rho^m: \Gamma \to \SL(\Sym^mV)
\]
denote the composition $\tau_m \circ \rho$. We have two main goals in
this section. The first is to prove that the representation $\rho^m$
is always EGF. This result does \emph{not} follow directly from the
fact that convex cocompact representations in $\P(V)$ are EGF
(\Cref{thm:convex_cocompact_egf}), because we do not know that the
representations $\rho^m$ are convex cocompact in $\P(\Sym^mV)$. In
fact, Jeff Danciger and Fanny Kassel have indicated in personal
communication to the author that $\rho^m$ should \emph{never} be
convex cocompact in $\P(\Sym^mV)$ unless the collection $\mc{H}$ is
empty: while $\rho^m(\Gamma)$ does preserve a properly convex domain
in $\P(\Sym^mV)$, the convex hull of the full orbital limit set in any
such domain seems ``too big'' for $\rho^m(\Gamma)$ to act cocompactly.

Our second goal is to show that the \emph{entire} space
$\Hom(\Gamma, \SL(\Sym^mV))$ is peripherally stable about $\rho^m$,
meaning (by \Cref{thm:stability_theorem}) an open neighborhood of
$\rho^m$ consists of EGF representations. In particular this shows
that small perturbations of $\rho^m$ still have finite kernel and
discrete image, giving new examples of discrete subgroups of
higher-rank Lie groups which are stable (as discrete groups).

\subsubsection{Proof strategy}

To show that $\rho^m$ is EGF, we will give an explicit description of
the boundary set $\hat{\Lambda}_m \subset \symflags(\Sym^mV)$. The
naive choice is to just take $\hat{\Lambda}_m$ to be
$\hat{\iota}(\hat{\Lambda}_\Omega(\Gamma))$, where
$\hat{\Lambda}_\Omega(\Gamma)$ is the flag-valued full orbital limit
set giving the EGF boundary set for $\rho:\Gamma \to \SL(V)$ (see
\Cref{sec:convex_cocompact}). While there is a equivariant surjective
map from this set to $\bgamh$, it turns out that we will have to
\emph{enlarge} it in order to ensure that the relevant dynamics hold.

The idea is the following: for each parabolic point $p \in \bgamh$,
with stabilizer $H$, we take the fiber over $p$ in $\hat{\Lambda}_m$
to be the space of flags in the boundary of a simplex
$S_H \subset \P(\Sym^mV)$, constructed using the simplex
$\Delta_H \subset \Omega$ on which $H = \Stab_\Gamma(p)$ acts
cocompactly. The simplex $S_H$ is chosen so that if $\gamma_n$ is a
sequence in $\Gamma$ converging to $p$, then a face of $S_H$ spans a
minimal attracting subspace of $\rho^m(\gamma_n)$. (Here, and
throughout this section, \emph{attracting} and \emph{repelling}
subspaces are understood in the sense defined in
\Cref{sec:flag_dynamics}).

We also want to ensure that the simplex $S_H$ is \emph{stable},
i.e. if $\rho^m_t$ is a small deformation of $\rho^m$ in
$\Hom(\Gamma, \SL(\Sym^mV))$, then $\rho_t^m(H)$ preserves a simplex
$S_H^t$ close to $S_H$. We verify that $S_H$ has these properties by
analyzing the relationship between the weights of $\rho$ and $\rho^m$
on the virtually abelian group $H$.

The other main steps in the proof involve checking that the boundary
set $\hat{\Lambda}_m$ we construct is actually a compact space
surjecting continuously onto $\bgamh$, and constructing the open sets
$\hat{C}_p$ also required by the definition of an EGF
representation. For the latter, we make heavy use of the fact that the
\emph{dual} action of $\rho(\Gamma)$ on $\P(V^*)$ is also projectively
convex cocompact, which allows us to construct a stable \emph{dual}
simplex $S_H^*$ for each $H \in \mc{H}$. The vertices of $S_H^*$ are
thought of as hyperplanes in $\P(\Sym^m(V))$, cutting out a region
$C_p$ of $\P(\Sym^m(V))$ on which $\rho^m(H)$ attracts points towards
$S_H$.

\subsubsection{Example: symmetric squares of convex projective
  3-manifold groups}

We illustrate the general idea of our approach with a specific
example. Let $\Omega \subset \P(\R^4)$ be a properly convex domain,
and let $\Gamma \subseteq \Aut(\Omega)$ be a discrete group acting
cocompactly on $\Omega$. In \cite{benoist2006convexes}, Benoist
produced examples of such groups which are hyperbolic relative to a
nonempty collection $\mc{H}$ of virtually abelian subgroups of rank
2. Further examples were constructed by Ballas-Danciger-Lee in
\cite{bdl2015convex}. Up to finite index, each $H \in \mc{H}$ acts
diagonalizably on $\P(\R^4)$, preserving a projective tetrahedron
$T_H \subset \P(\R^4)$ and acting cocompactly on a properly embedded
triangle $\Delta_H \subset \Omega$. Each edge of $\Delta_H$ is
contained in a unique supporting hyperplane of $\Omega$. The common
intersection of these hyperplanes is the fourth vertex of $T_H$.

More explicitly, up to finite index, each $H$ acts diagonally on
$\R^4$ in the basis $\{v_1, v_2, v_3, v_4\}$, where the $v_i$ are the
vertices of $T_H$, and $v_1, v_2, v_3$ are the vertices of
$\Delta_H$. We can consider the situation where (in this basis) $H$ is
the discrete group
\[
  \left\{
    \begin{psmallmatrix}
      2^a\\& 2^b\\ & & 2^c\\ & & & 1
    \end{psmallmatrix} : a,b,c \in \Z, \,a + b + c = 0 \right\}.
\]

The dual of $H$ preserves the corresponding dual basis
$\{v_1^*, v_2^*, v_3^*, v_4^*\}$, and acts cocompactly on a projective
triangle $\Delta_H^* \subset \P((\R^4)^*)$ with vertices
$v_1^*, v_2^*, v_3^*$. The kernels $\hyp{v_i^*}$ for $i = 1,2,3$ give
three supporting hyperplanes of $\Omega$ which cut out a region $R_H$
of projective space containing $\Omega$. In fact, $R_H$ also contains
$\dee \Omega - \dee\Delta_H$.

Now let $\rho^2:\Gamma \to \SL(\Sym^2(\R^4)) \simeq \SL(10, \R)$ be
the composition of the inclusion $\Gamma \hookrightarrow \SL(4, \R)$
with the symmetric square $\tau_2:\SL(4, \R) \to
\SL(\Sym^2(\R^4))$. In this case, the induced map
$\iota:\P(\R^4) \to \P(\R^{10})$ is the Veronese embedding.

For each $H \in \mc{H}$, $\rho^2(H)$ preserves a 9-simplex in
$\P(\R^{10})$, with vertices
\[
  \{v_1^2, v_2^2, v_3^2, v_4^2, v_1v_2, v_1v_3, v_1v_4, v_2v_3,
  v_2v_4, v_3v_4\}.
\]

In particular $\rho^2(H)$ also preserves the 5-simplex $S_H$ with
vertices
\[
  \{v_1^2, v_2^2, v_3^2, v_1v_2, v_2v_3, v_1v_3\}.
\]
Any divergent sequence in $\rho^2(H)$ always has an attracting
subspace spanned by a face of $S_H$, since the eigenvalues of elements
of $H$ on $v_4$ are always dominated by some eigenvalue of that
element on either $v_1, v_2$, or $v_3$. For instance, if we consider
the sequence
\[
  a_n =
  \begin{psmallmatrix}
    2^{2n}\\ & 2^{-n} \\ & & 2^{-n} \\ & & & 1
  \end{psmallmatrix},
\]
then $\rho^2(a_n)$ attracts towards the subspace spanned by
$\{v_1^2\}$. On the other hand, if $a_n$ is the sequence
\[
  \begin{psmallmatrix}
    2^n\\ & 2^n \\ & & 2^{-2n} \\ & & & 1
  \end{psmallmatrix},
\]
then $\rho^2(a_n)$ attracts towards the subspace spanned by
$\{v_1^2, v_2^2, v_1v_2\}$.

Moreover, the subspaces spanned by faces of $S_H$ are transverse to
$\iota(\dee \Omega - \dee \Delta_H)$, so large elements of $H$ attract
points in $\iota(\dee \Omega)$ that are ``far from'' $\dee S_H$
towards $\dee S_H$. In fact, this dynamical behavior extends to an
entire open subset of $\P(\R^{10})$: namely, a region cut out by the
hyperplanes corresponding to the vertices of the dual 5-simplex
$S_H^* \subset \P((\R^{10})^*)$ with vertices
\[
  \{(v_1^*)^2, (v_2^*)^2, (v_3^*)^2, v_1^*v_2^*, v_2^*v_3^*,
  v_1^*v_3^*\}.
\]
So, the simplex $S_H$ serves as the ``parabolic point'' for the action
of the peripheral subgroup $H$ on $P(\R^{10})$---and moreover, this
behavior is stable under perturbations of $\rho^2(H)$ in
$\Hom(\Gamma, \SL(\Sym^2V))$. This means that we can construct our
candidate boundary set for the representation $\rho^m$ by taking
\[
  \hat{\iota}(\hat{\Lambda}_\Omega(\Gamma)) \cup \bigcup_{H \in
    \mc{H}} \dee\hat{S}_H,
\]
where $\dee\hat{S}_H \subset \projflags(\Sym^2\R^4)$ is a closed
subset of the space of flags projecting to the boundary of the simplex
$S_H$.

\subsection{Generalized weight spaces}

To carry out the general construction of the simplex $S_H$ identified
in the previous example, we need some description of \emph{attracting
  subspaces} for the groups $\rho^m(H) \subset \SL(\Sym^mV)$. We
obtain this description by recalling some of the properties of
\emph{weights} of representations of free abelian groups.

\begin{definition}
  Let $\rho:H \to \GL(V)$ be a representation of a free abelian group
  $H$, and let $\rho_\C:H \to \GL(V \ten_\R \C)$ be the
  complexification of $\rho$.

  A \emph{complex weight} of $\rho$ is a homomorphism
  $\mu_\C : H \to \C$ such that the \emph{weight space}
  \[
    V_{\mu_\C} = \ker(\rho_\C(h) - \exp(\mu_\C(h))I)
  \]
  is nontrivial for evey $h \in H$. A \emph{generalized complex
    weight} is similarly a homomorphism $\mu_\C:H \to \C$ such that
  the \emph{generalized weight space}
  \[
    V_{\mu_\C} = \bigcup_{n=1}^{\dim V} \ker(\rho_\C(h) -
    \exp(\mu_\C(h))I)^n
  \]
  is nontrivial.

  For any generalized weight $\mu_{\C}$, the \emph{nilpotence degree}
  of $\mu_\C$ is the minimal $\ell \in \N$ such that
  $V_{\mu_\C} = \ker(\rho_\C(h) - \exp(\mu_\C(h))I)^\ell$.
\end{definition}

Given a representation $\rho:H \to \GL(V)$, the generalized complex
weight spaces of $\rho_\C$ give a $\rho_\C$-invariant decomposition of
$V \ten_\R \C$. This in turn gives a $\rho$-invariant decomposition of
$V$, since when $\mu_\C$ is a weight which takes on complex values,
the direct sum $V_{\mu_\C} \oplus V_{\overline{\mu_\C}}$ is a
$\rho$-invariant real subspace of $V$. By a slight abuse of
terminology we refer to this as the \emph{generalized weight space
  decomposition} for the representation $\rho$. The associated
\emph{real weights} of the representation are homomorphisms
$\mu:H \to \R$ of the form $\log |\exp \mu_\C |$, where $\mu_\C$ is a
(generalized) complex weight. For the rest of the section, unless
otherwise indicated, when we refer to (generalized) \emph{weights} of
a representation into $\SL(V)$, we will mean the (generalized)
\emph{real} weights, and similarly for weight spaces.

Generalized weight spaces of $\rho$ are stable under deformations of
$\rho$. To be precise, we observe the following:
\begin{prop}
  \label{prop:weights_deform_continuously}
  Let $\rho:\Gamma \to \GL(V)$ be a representation of a free abelian
  group, and let
  \[
    V_{\mu_1}^0 \oplus \ldots \oplus V_{\mu_s}^0
  \]
  be the generalized weight space decomposition of $V$ for $\rho$. Let
  $\rho_t$ be a continuous family of representations in
  $\Hom(\Gamma, \GL(V))$, with $\rho = \rho_0$.

  For all sufficiently small $t$, there is a $\rho_t$-invariant
  decomposition
  \[
    V_1^t \oplus \ldots \oplus V_s^t
  \]
  such that $V_i^t$ is a sum of generalized weight spaces for
  $\rho_t$, with $V_i^t$ varying continuously with $t$, and each of
  the weights associated to $V_i^t$ also varying continuously with
  $t$.
\end{prop}
\begin{proof}
  The weights vary continously as a set with multiplicity, because the
  roots of the characteristic polynomial of $\rho(\gamma)$ vary
  continuously in $\rho$ for fixed $\gamma \in \Gamma$. And, if $\mu$
  is a complex weight with multiplicity $k$, then for small $t$ there
  are complex weights $\mu_1^t, \ldots, \mu_k^t$ of $\rho_t$ (possibly
  with repeats) close to $\mu$ such that the sum of the kernels
  $\ker(\rho_\C(\gamma) - \exp(\mu_i^t(\gamma))I)^{\dim V}$ is close to
  $\ker(\rho_\C(\gamma) - \exp(\mu(\gamma))I)^{\dim V}$.
\end{proof}

\begin{definition}
  Let $\rho:H \to \SL(V)$ be a representation of a free abelian group
  $H$, and let $\Phi$ be the set of generalized weights of $\rho$. For
  any subset $\theta \subseteq \Phi$, we let $V_\theta \subseteq V$
  denote the span of the generalized weight spaces $V_\mu$ for
  $\mu \in \theta$, and we let $V_\theta^\opp \subseteq V$ denote the
  span of the generalized weight spaces $V_{\mu'}$ for
  $\mu' \in \Phi - \theta$.
\end{definition}

\subsubsection{Faces in the convex hull of the weights}

Whenever $\rho:H \to \SL(V)$ is a representation of a free abelian
group with rank $k$, we can extend any real (generalized) weight
$\mu:H \to \R$ to a homomorphism $\mu : H \ten_\Z \R \to \R$, and view
it as an element of $(\R^k)^* \simeq \R^k$.

\begin{definition}
  Let $\rho:H \to \SL(V)$ be a representation of a free abelian group,
  and let $\Phi$ be the set of generalized weights of $\rho$. We
  denote the closed convex hull of $\Phi$ in
  $(H \ten_\Z \R)^* \simeq (\R^k)^*$ by $\mc{C}(\rho)$; since $\Phi$
  is a finite subset of $(\R^k)^*$, $\mc{C}(\rho)$ is a convex
  polytope in $(\R^k)^*$.
\end{definition}

The convex polytope $\mc{C}(\rho)$ is important for our purposes
because it tells us how to find attracting subspaces for $\rho(H)$. In
particular, there is a correspondence between the \emph{faces} of
$\mc{C}(\rho)$ and attracting subspaces of $\rho(H)$.
\begin{definition}
  Let $\rho:H \to \SL(V)$ be a represention of a free abelian group
  with generalized weight set $\Phi$. Let $F$ be a closed face of
  $\mc{C}(\rho)$. We let $\Phi(F)$ denote the set of generalized
  weights of $\rho$ lying in $F$.

  For a face $F$ of $\mc{C}(\rho)$, we write $V_F$ and $V_F^\opp$ for
  $V_{\Phi(F)}$ and $V_{\Phi(F)}^\opp$, respectively.
\end{definition}

Below, we will prove the following:

\begin{prop}
  \label{prop:boundary_weights_determine_dynamics}
  Let $\rho:H \to \SL(V)$ be a representation of a free abelian
  group. For any divergent sequence $h_n \in H$, there is a face $F$
  of $\mc{C}(\rho)$ such that $V_F$ and $V_F^\opp$ are respectively
  attracting and repelling subspaces for $\rho(h_n)$.

  Conversely, for any face $F$ of $\mc{C}(\rho)$, $V_F$ is an
  attracting subspace for some sequence $\rho(h_n)$ with $h_n \in H$
  divergent.
\end{prop}

To prove \Cref{prop:boundary_weights_determine_dynamics}, we first
establish some estimates which will later help us show that the
convergence to the spaces $V_F$ is both uniform and stable. To help
make our estimates explicit, we choose a norm $|\cdot|$ on
$H \ten_\Z \R \simeq \R^k$. We also fix a norm $||\cdot||$ on $\C^d$,
which induces an operator norm $||\cdot||$ on $\SL(d, \C)$. We use
$\mathbf{m}(\cdot)$ to denote the conorm
$\mathbf{m}(g) = ||g^{-1}||^{-1}$.
\begin{lem}
  \label{lem:upper_triangular_norm_bound}
  Let $U(d, \C)$ be the group of upper-triangular matrices in
  $\SL(d, \C)$, and let $H$ be a free abelian group. For any
  $\rho \in \Hom(H, U(d, \C))$, there exists $D(\rho) > 0$ (varying
  continuously with $\rho$) so that for any $h \in H$, we have
  \[
    \frac{1}{D}|h|^{1-d} \cdot r^-_\rho(h) \le \mathbf{m}(\rho(h)) \le
    ||\rho(h)|| \le D|h|^{d-1} \cdot r^+_\rho(h),
  \]
  where $r^+_\rho(h)$ and $r^-_\rho(h)$ are respectively the maximum
  and minimum modulus of any eigenvalue of $\rho(h)$.
\end{lem}
\begin{proof}
  We can extend each $\rho \in \Hom(H, U(d, \C))$ to a representation
  $\rho:H \ten_\Z \R \to U(d, \C)$.  For each $h \in H$, we write
  \[
    \rho(h) = S^\rho(h) + N^\rho(h),
  \]
  where $S^\rho(h)$ is diagonal and $N^\rho(h)$ is strictly
  upper-triangular. The entries of $S^\rho(h)$ are the eigenvalues of
  $\rho(h)$, which means $||S^\rho(h)|| = r^+_\rho(h)$ and
  $||S^\rho(-h)||^{-1} = r^-_\rho(h)$.

  For any fixed $h \in H \ten_\Z \R$ and $k \in \N$, if we binomially
  expand the expression
  \begin{align*}
    \rho(h)^k &= (S^\rho(h) + N^\rho(h))^k,
  \end{align*}
  then all monomials containing at least $d$ terms equal to
  $N^\rho(h)$ must vanish. This gives us an inequality of the form
  \begin{equation}
    \label{eq:norm_upperbound}
    ||\rho(h)^k|| \le ||S^\rho(h)||^k \cdot p(k, ||N^\rho(h)||),
  \end{equation}
  where $p(x, y)$ is a polynomial with nonnegative coefficients
  depending continuously on $\rho$ and $h$, with degree at most
  $d - 1$. Using the identity
  $\mathbf{m}(\rho(h)^k) = ||\rho(-h)^k||^{-1}$, we also see that
  \begin{equation}
    \label{eq:norm_lowerbound}
    \mathbf{m}(\rho(h)^k) \ge ||S^\rho(-h)||^{-k} / p(k,
    ||N^\rho(-h)||).
  \end{equation}
  Now, since the unit sphere in $H \ten_\Z \R$ is compact, and
  $N^\rho(h)$ varies continuously with $\rho$ and $h$, we can find
  $D_0$ varying continuously with $\rho$ so that
  $||N^\rho(s)|| \le D_0$ for all $s$ with $|s| = 1$. Now suppose that
  $h = ks$ for $k \in \N$ and $s \in H \ten_\Z \R$ with $|s| =
  1$. Then $k = |h|$ and from \eqref{eq:norm_upperbound} we see
  \[
    ||\rho(h)|| = ||\rho(s)^k|| \le r^+_\rho(s)^k \cdot p(|h|, ||N^\rho(s)||) =
    r^+_\rho(h) \cdot p(|h|, ||N^\rho(s)||).
  \]
  Since $0 \le ||N^\rho(s)|| < D_0$, we can bound the polynomial term
  beneath $D|h|^{d-1}$ for a constant $D$ depending continuously on
  $\rho$, giving us the desired upper bound. For the lower bound, we
  can argue similarly using \eqref{eq:norm_lowerbound}. A priori,
  these bounds only holds for $h \in H \ten_\Z \R$ which are positive
  integer multiples of elements on the unit sphere, but every
  $h \in H$ is bounded distance from such an element, so we get the
  desired bounds everywhere.
\end{proof}

For the next estimate, we choose an inner product on our real vector
space $V$, which induces a norm $|| \cdot ||$ on $V$ and a smooth
metric $d_\P$ on $\P(V)$. Specifically, for any transverse subspaces
$W, W' \subset V$, we define
\[
  \angle(W, W') = \inf_{\substack{w \in W - \{0\},\\w' \in W' -
      \{0\}}} \angle(w, w'),
\]
and then take $d_\P([u], [v]) = \sin(\angle([u],[v]))$. As before, the
norm $||\cdot||$ induces an operator norm and conorm
$\mathbf{m}(\cdot)$ on $\SL(V)$.
\begin{lem}
  \label{lem:norm_attract_estimate}
  Let $W, W^\perp$ be transverse subspaces of $V$ with
  $W \oplus W^\perp = V$. If $g \in \SL(V)$ preserves both $W$ and
  $W^\perp$, then for any $x \in \P(V) - \P(W^\perp)$, we have
  \[
    \frac{d_\P(g \cdot x, \P(W))}{d_\P(g \cdot x, \P(W^\perp))} \le
    \frac{1}{\sin^2 \angle(W,
      W^\perp)}\frac{||g|_{W^\perp}||}{\mathbf{m}(g|_W)} \cdot d(x,
    \P(W^\perp))^{-1}.
  \]
\end{lem}
\begin{proof}
  Let $x = [v]$ for $v \in V$, and uniquely write $v = w + w^\perp$
  for $w \in W$, $w^\perp \in W^\perp$, with $w^\perp \ne 0$. Then we
  have
  \[
    \frac{||w^\perp||}{||w||}\sin \angle(W, W^\perp) \le d_\P(x,
    \P(W)) \le \frac{||w^\perp||}{||w||},
  \]
  and similarly
  \[
    \frac{||w||}{||w^\perp||}\sin \angle(W, W^\perp) \le d_\P(x,
    \P(W^\perp)) \le \frac{||w||}{||w^\perp||}.
  \]
  So in particular we have
  \begin{equation}
    \label{eq:subspace_dist_estimate}
    \sin \angle(W, W^\perp) \frac{||w^\perp||}{||w||} \le
    \frac{d_\P(x, \P(W))}{d_\P(x, \P(W^\perp))} \le \frac{1}{\sin
      \angle(W, W^\perp)} \frac{||w^\perp||}{||w||}.
  \end{equation}
  Since $g$ preserves the decomposition $V = W \oplus W^\perp$, we see
  that
  \[
    \frac{d_\P(g \cdot x, \P(W))}{d_\P(g \cdot x, \P(W^\perp))} \le
    \frac{1}{\sin \angle(W, W^\perp)} \frac{||g \cdot w^\perp||}{||g
      \cdot w||}.
  \]
  We know that $||g \cdot w^\perp|| \le ||g|_{W^\perp}|| \cdot
  ||w^\perp||$ and $||g \cdot w|| \ge \mathbf{m}(g|_W) \cdot ||w||$,
  so we get the inequality
  \[
    \frac{d_\P(g \cdot x, \P(W))}{d_\P(g \cdot x, \P(W^\perp))} \le
    \frac{1}{\sin \angle(W, W^\perp)}
    \frac{||g|_{W^\perp}||}{\mathbf{m}(g|_W)} \cdot
    \frac{||w^\perp||}{||w||}.
  \]
  Then we apply the left-hand inequality of
  \eqref{eq:subspace_dist_estimate} and use the fact that
  $d_\P(x, \P(W)) \le 1$ to finish the proof.
\end{proof}

\begin{proof}[Proof of
  \Cref{prop:boundary_weights_determine_dynamics}]
  Let $\rho:H \to \SL(V)$ be a representation of a free abelian group,
  let $\Phi$ be the set of generalized weights, and let $h_n$ be a
  divergent sequence in $H$. Up to subsequence, the sequence
  $h_n / |h_n|$ converges to some $h_\infty \in H \ten_\Z \R$ with
  $|h_\infty| = 1$.

  We can view $h_\infty$ as a linear functional on the space
  $(H \ten_\Z \R)^*$. Since $\mc{C}(\rho) \subset (H \ten_\Z \R)^*$ is
  a convex polytope, this means there is a face $F$ of $\mc{C}$ so
  that for any $\mu \in \Phi(F)$ and any
  $\mu^\opp \in \Phi - \Phi(F)$, we have
  $\mu(h_\infty) > \mu^\opp(h_\infty)$. Then for sufficiently large
  $n$ we also have $\mu(h_n/|h_n|) > \mu^\opp(h_n/|h_n|)$. In fact,
  since $\Phi$ is finite, there is a constant $M > 0$ such that
  $\mu(h_n) - \mu^\opp(h_n) > M|h_n|$ for every $\mu \in \Phi(F)$ and
  every $\mu^\opp \in \Phi - \Phi(F)$. In particular, if $r(h_n)$ is
  the eigenvalue of $\rho(h_n)$ on $V_F$ with smallest modulus, and
  $r^\opp(h_n)$ is the eigenvalue of $\rho(h_n)$ on $V_F^\opp$ with
  largest modulus, we have $|r(h_n)| / |r^\opp(h_n)| > \exp(M|h_n|)$.

  We can choose an identification of $V \ten_\R \C$ with $\C^{\dim V}$
  so that the complexification
  $\rho_\C:H \ten_\Z \R \to \SL(V \ten_\R \C)$ lies in the group of
  upper-triangular matrices, and the eigenvectors of $\rho_\C$ are
  standard basis vectors. The norm $||\cdot||$ we have chosen on $V$
  induces a norm on $V \ten_\R \C$ which agrees with the standard norm
  on $\C^{\dim V}$ up to bounded multiplicative error. So, we can
  apply \Cref{lem:upper_triangular_norm_bound} to see that the
  quantity
  \[
    \frac{\mathbf{m}(\rho(h_n)|_{V_F})}{||\rho(h_n)|_{V_F^\opp}||}
  \]
  tends to infinity as $n \to \infty$. Then
  \Cref{lem:norm_attract_estimate} implies that for any
  $x \in \P(V) - \P(V_F^\opp)$, the distance
  \[
    d_\P(\rho(h_n)x, \P(V_F)) \le \frac{d_\P(\rho(h_n)x,
      \P(V_F))}{d_\P(\rho(h_n)x, \P(V_F^\opp))}
  \]
  tends to $0$ as $n \to \infty$, so $V_F$ and $V_F^\opp$ must
  respectively be attracting and repelling subspaces for $\rho(h_n)$.

  Conversely, if $F$ is any face of $\mc{C}(\rho)$, we can choose
  $h \in H \ten_\Z \R$ so that $\mu(h) > 0$ and $\mu(h) > \mu^\opp(h)$
  for any $\mu \in \Phi(F)$ and $\mu^\opp \in \Phi - \Phi(F)$. Then if
  $h_n \in H$ is any divergent sequence with $h_n / |h_n| \to h$ in
  $H \ten_\Z \R$, we can similarly apply
  \Cref{lem:upper_triangular_norm_bound} to see that the ratio
  $||\rho(h_n)|_{V_F}|| / ||\rho(h_n)|_{V_F^\opp}||$ tends to
  infinity, and again use \Cref{lem:norm_attract_estimate} to see that
  $V_F$ is an attracting subspace for $h_n$.
\end{proof}

\subsection{Weights of peripheral subgroups in convex cocompact
  groups}

For the rest of this section, we fix a relatively hyperbolic pair
$(\Gamma, \mc{H})$, where each $H \in \mc{H}$ is virtually abelian
with rank at least 2. We also fix a representation
$\rho:\Gamma \to \SL(V)$ which is convex cocompact in $\P(V)$, and let
$\Omega \subset \P(V)$ be a properly convex domain where
$\rho(\Gamma)$ acts convex cocompactly.

Our goal now is to describe the convex polytope in $(H \ten_\Z \R)^*$
associated to the restriction of $\rho$ to each $H \in \mc{H}$, which
we can use to understand the dynamics of both $\rho(H)$ and
$\rho^m(H)$.

\begin{definition}
  For each $H \in \mc{H}$, we let $\verts{H} \subset \P(V)$ denote the
  set of vertices of $\Delta_H$.
\end{definition}

\begin{prop}
  \label{prop:vertices_weight_spaces}
  Let $H \in \mc{H}$ be a peripheral subgroup of rank $k \ge 2$, and
  let $H_0 \subseteq H$ be a finite-index free abelian
  subgroup. Consider the restriction $\rho_0 = \rho|_{H_0}$. Then, the
  convex polytope $\mc{C}(\rho_0)$ is a $k$-simplex in
  $(H_0 \ten_\Z \R)^*$, and each vertex of $\mc{C}(\rho_0)$ is a
  weight $\mu$ whose associated weight space is a vertex of
  $\Delta_H$.

  Moreover, every weight of $\rho_0$ which is not a vertex of
  $\mc{C}(\rho_0)$ lies in the interior of $\mc{C}(\rho_0)$.
\end{prop}
\begin{proof}
  Each vertex $v \in \verts{H}$ lies in a weight space of $\rho_0$,
  with an associated weight $\mu_v$. Let $\Phi$ denote the weights of
  $\rho_0$, and let $\Phi(\verts{H}) \subseteq \Phi$ be the set of
  weights of the form $\mu_v$ for $v \in \verts{H}$. We claim that for
  any $\mu \in \Phi - \Phi(\verts{H})$ and any $h \in H_0 \ten_\Z \R$,
  there is a vertex $v \in \verts{H}$ such that
  \[
    \mu_v(h) > \mu(h).
  \]
  Suppose for a contradiction that the claim does not hold, i.e. there
  exists $h \in H_0 \ten_\Z \R$ and $\mu \in \Phi - \Phi(\verts{H})$
  such that $\mu(h) \ge \mu_v(h)$ for all $v \in \verts{H}$. We choose
  a subspace $V_\mu'$ of the weight space $V_\mu$, so that the
  restriction of $\rho_0$ to $V_\mu'$ is (complex)
  diagonalizable. Then we let $W_H = \spn(\Delta_H)$, and let $V'$ be
  the $\rho_0$-invariant subspace $V_\mu' \oplus W_H$. Then $\rho_0$
  induces a representation $\rho_0':H_0 \to \SL(V')$. Since each
  vertex in $\mc{V}_H$ is an eigenspace for $\rho_0$, this
  representation is (complex) diagonalizable.

  We may choose our norm on $V'$ so that the eigenspaces for $\rho_0'$
  are pairwise orthogonal. Then, for any $h \in H_0$, the norm of
  $\rho_0'(h)$ restricted to any weight space is given by the modulus
  of the corresponding weight. So we have
  \[
    ||\rho_0(h)|_{W_H}|| \le \mathbf{m}(\rho_0(h)|_{V_\mu'}).
  \]
  Now, if $h_n$ is any divergent sequence in $H$ with
  $h_n/|h_n| \to h$, \Cref{lem:norm_attract_estimate} implies that for
  any $x \in \P(V') - \P(W_H)$, the ratio
  \[
    \frac{d_\P(\rho(h_n)x, \P(V_\mu))}{d_\P(\rho(h_n)x, \P(W_H))}
  \]
  does not tend to infinity as $n \to \infty$. In particular this is
  true for some $x \in \Omega$, since $\P(V') \cap \Omega$ is
  relatively open and nonempty. But since $\Delta_H \subset \P(W_H)$,
  this contradicts the fact that $\dee \Delta_H$ is the full orbital
  limit set of $\rho(H_0)$ in $\Omega$.

  We have now proved our claim, which implies that any extreme point
  of the convex polytope $\mc{C}(\rho_0)$ is a weight $\mu_v$ for
  $v \in \verts{H}$. On the other hand, by
  \Cref{cor:segments_in_peripherals}, we may assume that each vertex
  $v \in \verts{H}$ is an extreme point in $\dee \Omega$, and by
  \Cref{prop:attracting_spaces_in_faces}, this means that for each
  $v \in \verts{H}$, there is a sequence $h_n \in H_0$ such that $v$
  is an attracting subspace for $\rho_0(h_n)$. Since $\rho_0(H_0)$
  acts diagonalizably on $W_H$, this implies that $\mu_v$ is an
  extreme point of $\mc{C}(\rho_0)$. The last assertion of the
  proposition follows directly from the claim.
\end{proof}

\Cref{prop:vertices_weight_spaces} tells us that we can
combinatorially identify the $k$-simplex $\Delta_H$ and the
$k$-simplex $\mc{C}(\rho_0)$ for $\rho_0 = \rho|_{H_0}$. We write this
identification explicitly:
\begin{definition}
  Let $H \in \mc{H}$, and let $H_0$ be a finite-index free abelian
  subgroup. For each face $F$ of $\Delta_H$ with vertices $\mc{V}(F)$,
  we let $\rface{F}$ denote the face of $\mc{C}(\rho_0)$ whose
  vertices are the weights $\mu_v$ for $v \in \mc{V}(F)$.
\end{definition}

\subsection{Invariant simplices in the symmetric
  power}\label{subsec:symmetric_power_simplices}

Our next step is to describe the simplices $S_H \subset \P(\Sym^mV)$
which give rise to the fibers in $\hat{\Lambda}_m$ over parabolic
points, for our EGF boundary extension $\hat{\Lambda}_m \to \bgamh$.

Let $H \in \mc{H}$ have rank $k$, and let $H_0 \subseteq H$ be a
finite-index free abelian subgroup. We let $\rho_0, \rho_0^m$
respectively denote the restrictions of $\rho, \rho^m$ to $H_0$, and
let $\Phi, \Phi^m$ denote the sets of weights of $\rho_0$ and
$\rho_0^m$. We observe the following:
\begin{lem}
  \label{lem:symmetric_polytyope_simplex}
  The convex polytope $\mc{C}(\rho_0^m)$ is the $k$-simplex
  $m\mc{C}(\rho_0)$. Moreover, for every face $\rface{F}$ of
  $\mc{C}(\rho_0)$, the weights in $\Phi^m \cap m\rface{F}$ are
  exactly the vertices of the $m$th barycentric subdivision of
  $m\rface{F}$, and each such weight has a one-dimensional generalized
  weight space.
\end{lem}
\begin{proof}
  The weights of $\rho_0^m$ are exactly the set of homomorphisms of
  the form
  \[
    \sum_{\mu \in \Phi} a_\mu \mu,
  \]
  where $a_\mu \in \N \cup \{0\}$ and $\sum a_\mu = m$. In particular,
  the set of rescaled weights $\frac{1}{m}\Phi^m$ consists entirely of
  convex combinations of weights of $\rho_0$, and contains every
  weight in $\Phi$. This (together with
  \Cref{prop:vertices_weight_spaces}) implies that $\mc{C}(\rho_0^m)$
  is a $k$-simplex.

  Further, every (rescaled) weight in the boundary of the rescaled
  simplex $\frac{1}{m}\mc{C}(\rho_0^m)$ must be a convex combination
  of weights lying in a single face of the simplex
  $\mc{C}(\rho_0)$. But \Cref{prop:vertices_weight_spaces} says that
  every weight in $\Phi \cap \dee \mc{C}(\rho_0)$ is a vertex of
  $\mc{C}(\rho_0)$. So, if $F$ is a face of the simplex $\Delta_H$
  with vertices $\mc{V}(F)$, the weights in
  $\rface{F} \cap \frac{1}{m}\Phi^m$ are exactly the convex
  combinations of the form
  \begin{equation}
    \label{eq:convex_combination_weights}
    \frac{1}{m}\sum_{v \in \mc{V}(F)} a_v \mu_v,
  \end{equation}
  where $a_v \in \N \cup \{0\}$ and $\sum a_v = m$. These are exactly
  the vertices in the $m$th barycentric subdivision of $\rface{F}$,
  and in fact each such vertex has \emph{unique} expression of the
  form \eqref{eq:convex_combination_weights}. Since each weight
  $\mu_v$ for $v \in \mc{V}(F)$ has a one-dimensional generalized
  weight space, it follows that the weights in $\Phi^m \cap mF$ do as
  well.
\end{proof}

\subsubsection{The simplices $S_H \subset \P(\Sym^mV)$}

Using \Cref{lem:symmetric_polytyope_simplex}, we can define the
vertices of the simplex $S_H$: they are exactly the weight spaces for
the weights $\mu$ lying in $\Phi^m \cap \dee \mc{C}(\rho_0^m)$.

To define $S_H$ as a subset of $\P(\Sym^mV)$, we choose lifts in
$\Sym^mV$ of each vertex of $S_H$, and then take convex
combinations. Our lifts are chosen as follows: we first pick a lift
$\tilde{v} \in V$ of each vertex $v \in \verts{H}$, so that $\Delta_H$
is the projectivization of the convex hull in $V$ of
$\{\tilde{v} : v \in \verts{H}\}$. The weight space of $\mu$ for each
$\mu \in \dee \mc{C}(\rho_0^m)$ is spanned by a unique vector in
$\Sym^mV$ of the form
\[
  \tilde{v}_\mu = \prod_{v \in \verts{H}} \tilde{v}^{a_v},
\]
where $a_v \in \N \cup \{0\}$ and $\sum a_v = m$. Then we can define
$S_H$ to be the projectivization of the convex hull in $\Sym^mV$ of
the $\tilde{v}_\mu$'s.

\subsubsection{Dynamics on the simplices $S_H$}

By definition, the vertices of $S_H$ are exactly the weight spaces for
the weights in the boundary of the simplex
$\mc{C}(\rho_0^m) \subset (H_0 \ten_\Z \R)^*$. So,
\Cref{prop:boundary_weights_determine_dynamics} immediately implies
the following:
\begin{cor}
  \label{cor:symmetric_simplex_attracting_subspace}
  Let $H \in \mc{H}$. For every divergent sequence $h_n \in H$, there
  is a face $F$ of $S_H$ which spans an attracting subspace for the
  sequence $\rho^m(h_n)$.
\end{cor}

\subsection{Dual simplices in symmetric powers}

As discussed in \Cref{sec:convex_cocompact}, \cite[Proposition
5.6]{dgk2017convex} says that since $\rho:\Gamma \to \SL(V)$ is convex
cocompact in $\P(V)$, the dual representation $\Gamma \to \SL(V^*)$ is
convex cocompact in $\P(V^*)$, and in fact there is a domain
$\Omega \subset \P(V)$ so that $\Gamma$ acts convex cocompactly on
both $\Omega$ and the dual domain $\Omega^*$. By the work of
Islam-Zimmer \cite{iz2019convex}, each virtually abelian subgroup
$H \in \mc{H}$ must act cocompactly on a properly embedded \emph{dual}
simplex $\Delta_H^* \subset \Omega^*$. And, for each vertex $w$ of
$\Delta_H^*$, the projective hyperplane $\hyp{w}$ is a supporting
hyperplane of $\Omega$ at $\dee \Delta_H$.

\subsubsection{The simplices $S_H^* \subset \P(\Sym^mV^*)$}

For each $H \in \mc{H}$, we can define an $H$-invariant \emph{dual}
simplex $S_H^* \subset \P(\Sym^mV^*)$, by carrying out the
construction we used to find $S_H$ (but this time for the dual
representation $\rho^*:\Gamma \to \SL(V^*)$). We can describe the
relationship between the simplices $S_H$ and $S_H^*$ a little more
explicitly. For a finite-index free abelian subgroup
$H_0 \subseteq H$, we let $\rho_0^*:H_0 \to \SL(V^*)$ be the dual of
the restriction of $\rho$ to $H_0$, and similarly define
$(\rho_0^m)^*:H_0 \to \SL(V^*)$. Then the weights of $\rho_0^*$ are
the negative weights of $\rho_0$, and the weights of $(\rho_0^m)^*$
are the negative weights of $\rho_0^m$.

Suppose $\mu^m$ is a weight of $\rho_0^m$ with a one-dimensional
generalized weight space $v^m$. Then, the negative weight $-\mu^m$
also has a one-dimensional generalized weight space
$w^m \in \P(\Sym^mV^*)$, and $\hyp{w^m}$ is the hyperplane spanned by
the weight spaces of the weights in $\Phi^m - \{\mu^m\}$. In
particular, we can consider the case where $\mu^m$ is a weight lying
in the boundary of $\mc{C}(\rho_0^m)$. In this case, $v^m$ is a vertex
of $S_H$, $w^m$ is a vertex of $S_H^*$, and $\hyp{w^m}$ is a
hyperplane intersecting $S_H$ in a codimension-1 face of $S_H$.

This allows us to define a simultaneous lift of the \emph{boundaries}
of the simplices $S_H, S_H^*$ in the space of flags
$\symflags(\Sym^mV)$.
\begin{definition}
  For a peripheral subgroup $H \in \mc{H}$, we let $\dee \hat{S}_H$
  denote the set
  \[
    \dee \hat{S}_H = \{(v, w) \in \symflags(\Sym^m(V)) : v \in \dee
    S_H, w \in \dee S_H^*\}.
  \]
\end{definition}
The discussion above shows that $\dee \hat{S}_H$ is a nonempty closed
invariant subset of $\symflags(\Sym^mV)$, projecting to $\dee S_H$ and
$\dee S_H^*$ under the canonical maps
$\symflags(\Sym^mV) \to \P(\Sym^mV)$ and
$\symflags(\Sym^mV^*) \to \P(V^*)$.

\subsection{Defining the boundary set}

Using the sets $\dee \hat{S}_H$, we can define our candidate for the
EGF boundary set $\hat{\Lambda}_m \subset \symflags(\Sym^mV)$ as
follows. We let $\hat{\phi}:\hat{\Lambda}_\Omega(\Gamma) \to \bgamh$
denote the boundary extension for the EGF representation $\rho$ coming
from the proof of \Cref{thm:convex_cocompact_egf}. For each
$z \in \bgamh$, we define the set
$\hat{\psi}_m(z) \subset \symflags(\Sym^m(V))$ by:
\[
  \hat{\psi}_m(z) =
  \begin{cases}
    \hat{\iota}(\hat{\phi}^{-1}(z)), & z \in \conbdry(\Gamma, \mc{H})\\
    \dee\hat{S}_H, & z \in \parbdry(\Gamma, \mc{H}).
  \end{cases}
\]
We define
\[
  \hat{\Lambda}_m = \bigcup_{z \in \bgamh} \hat{\psi}_m(z),
\]
and observe that
$\hat{\iota}(\hat{\Lambda}_\Omega(\Gamma)) \subset \hat{\Lambda}_m$.

The set $\hat{\Lambda}_m$ is $\rho^m(\Gamma)$-invariant, since
$\hat{\iota}$ is $\tau_m$-equivariant and the construction of the set
$\dee \hat{S}_H$ is invariant. Ultimately we want to see that
$\hat{\Lambda}_m$ is compact, and that there is a well-defined
transverse map $\hat{\phi}_m:\hat{\Lambda}_m \to \bgamh$ giving us our
EGF boundary extension.

\subsection{Defining the boundary extension}

Our next immediate goal is to show:
\begin{prop}
  \label{prop:symmetric_square_map}
  For distinct $z_1, z_2 \in \bgamh$, the sets
  \[
    \hat{\psi}_m(z_1), \quad \hat{\psi}_m(z_2)
  \]
  are transverse (in particular, disjoint). Consequently, the map
  $\hat{\phi}_m:\hat{\Lambda}_m \to \bgamh$ given by
  \[
    \hat{\phi}_m(\xi) = z \iff \xi \in \hat{\psi}_m(z)
  \]
  is well-defined, equivariant, surjective, and transverse.
\end{prop}

\begin{lem}
  \label{lem:simplex_intersection_closed}
  Let $X$ be a closed subset of $\Lambda_\Omega(\Gamma)$, and let
  $\mc{H}_X \subset \mc{H}$ be the set
  $\{H \in \mc{H} : \dee \Delta_H \cap X \ne \emptyset\}$.

  Then the set
  \[
    X_{\mc{H}} = X \cup \bigcup_{H \in \mc{H}_X} \Delta_H
  \]
  is closed.
\end{lem}
\begin{proof}
  Let $x_n$ be a sequence in $X_{\mc{H}}$. By compactness of
  $\Lambda_\Omega(\Gamma)$, we can choose a subsequence so that
  $x_n \to x \in \Lambda_\Omega(\Gamma)$. We wish to show that
  $x \in X_{\mc{H}}$. Since $X$ is closed and $X \subset X_{\mc{H}}$,
  we may assume that for each $n$, we have $x_n \in \dee\Delta_{H_n}$
  for some $H_n \in \mc{H}_X$.

  Up to subsequence, the sets $\dee\Delta_{H_n}$ converge to a closed
  set $\dee\Delta_\infty$ which is a connected finite union of
  (possibly degenerate) projective simplices. We must have
  $x \in \dee \Delta_\infty \subset \Lambda_\Omega(\Gamma)$. By
  definition, $\dee \Delta_{H_n}$ intersects $X$ nontrivially, and
  since $X$ is closed we must also have
  $\dee \Delta_\infty \cap X \ne \emptyset$.

  Suppose for a contradiction that $x \notin X_{\mc{H}}$. Then in
  particular $x \notin X$. Since $\dee \Delta_\infty$ intersects $X$,
  it must contain at least two points, which means that every point in
  $\dee \Delta_\infty$ lies in a nontrivial closed projective segment
  (since $\dee \Delta_\infty$ is a connected finite union of
  projective simplices). But then by
  \Cref{cor:segments_in_peripherals},
  $\dee \Delta_\infty \subset \dee \Delta_H$ for some $H \in \mc{H}$,
  and since $\dee \Delta_\infty \cap X \ne \emptyset$ we have
  $H \in \mc{H}_X$ and therefore $x \in X_{\mc{H}}$, contradiction.
\end{proof}

\begin{prop}
  \label{prop:simplex_connected_comp}
  For each $H \in \mc{H}$, there is a connected subset $C_H$ in
  \[
    \Opp(\dee S_H^*) = \{x \in \P(\Sym^mV) : x \perp w \textrm{ for
      every } w \in \dee S_H^*\}
  \]
  such that for every closed subset
  $X \subset \Lambda_\Omega(\Gamma) - \Delta_H$, $C_H$ contains the
  closure of
  \[
    \iota(X) \cup \bigcup_{H' \in \mc{H}_X} S_{H'},
  \]
  where
  $\mc{H}_X = \{H \in \mc{H} : \dee \Delta_H \cap X \ne \emptyset\}$.
\end{prop}
\begin{proof}
  Let $\verts{H}$, $\verts{H}^*$ denote the vertex sets of $\Delta_H$
  and $\Delta_H^*$, respectively. Using the convexity of $\Omega$, we
  can find lifts $\tilde{w} \in V^*$ for each vertex
  $w \in \verts{H}^*$, a continuous lift $\tilde{\Lambda} \subset V$
  of $\Lambda_\Omega(\Gamma)$, and a continuous lift
  $\tilde{\Delta}_H \subset V$ of $\Delta_H$ so that
  \begin{equation}
    \label{eq:positivity_original_v}
    \tilde{w}(\tilde{\Lambda} - \tilde{\Delta}_H) > 0
  \end{equation}
  for every $w \in \verts{H}^*$.

  The lifts $\tilde{w}$ induce lifts $\tilde{w}^m \in (\Sym^mV)^*$ of
  each vertex $w^m$ of $S_H^*$. We take the set $C_H$ to be the
  projectivization of
  \[
    \{v \in \Sym^mV : \tilde{w}^m(v) > 0 \textrm{ for all vertices }
    w^m \textrm{ of } S_H^*\}.
  \]
  Every point in $\dee S_H^*$ is the projectivization of a convex
  combination of the lifts $\tilde{w}^m$. This tells us that $C_H$ is
  a connected subset of
  $\P(\Sym^mV) - \bigcup_{w \in \dee(S_H^*)} \hyp{w}$.

  Now let $X \subset \Lambda_\Omega(\Gamma) - \Delta_H$ be closed and
  let $Y$ be the set
  \[
    \iota(X) \cup \bigcup_{H' \in \mc{H}_X} S_{H'}.
  \]
  We wish to show that $\overline{Y} \subset C_H$. Let $X_{\mc{H}}$ be
  the set
  \[
    X_{\mc{H}} = X \cup \bigcup_{H' \in \mc{H}_X} \Delta_{H'}.
  \]
  By \Cref{lem:simplex_intersection_closed}, we can find a
  \emph{compact} lift $\tilde{X}_{\mc{H}}$ of $X_{\mc{H}}$ in $V$ so
  that $\tilde{w}(\tilde{x}) > 0$ for every
  $\tilde{x} \in \tilde{X}_{\mc{H}}$ and every $w \in \verts{H}^*$.
  We consider the set
  \[
    \Sym^m\tilde{X}_{\mc{H}} = \{\tilde{x}^m \in \Sym^mV : \tilde{x}^m
    = \prod_{i=1}^m \tilde{x}_i \textrm{ for } \tilde{x}_i \in
    \tilde{X}_{\mc{H}}\}.
  \]
  This set is the image of the $m$-fold Cartesian product
  $(\tilde{X}_{\mc{H}})^m$ under the continuous map $V^m \to \Sym^mV$
  given by $(v_1, \ldots, v_m) \mapsto v_1\cdots v_m$, so it is
  compact. Moreover, since $\Sym^m\tilde{X}_{\mc{H}}$ contains a lift
  of every vertex of every $S_{H'}$ for $H' \in \mc{H}_X$, the
  projectivization of the convex hull of $\Sym^m\tilde{X}_{\mc{H}}$
  contains $Y$, hence $\overline{Y}$.

  But from \eqref{eq:positivity_original_v}, we know that
  $\tilde{w}^m(\tilde{x}^m) > 0$ for every vertex $w^m$ of $S_H^*$ and
  every $\tilde{x}^m \in \tilde{X}_{\mc{H}}$, so we see that $C_H$
  contains the projectivization of the convex hull of
  $\Sym^m\tilde{X}_{\mc{H}}$.
\end{proof}

\begin{proof}[Proof of \Cref{prop:symmetric_square_map}]
  Let $z_1, z_2 \in \bgamh$ be distinct. If both $z_1$ and $z_2$ are
  conical limit points, the proposition follows from the
  transversality of the EGF boundary extension
  $\hat{\phi}:\hat{\Lambda}_\Omega(\Gamma) \to \bgamh$ and the fact
  that $\hat{\iota}$ preserves transversality. On the other hand, if
  $z_1$ is a parabolic point, this follows from
  \Cref{prop:simplex_connected_comp} (and the equivalent dual
  statement).
\end{proof}

\subsection{Dynamics on $S_H$}

We have now defined an equivariant transverse surjective map
$\hat{\phi}_m:\hat{\Lambda}_m \to \bgamh$, but we do not yet know that
the set $\hat{\Lambda}_m$ is compact, or even that $\hat{\phi}_m$ is
continuous. However, it turns out that it is easier to verify these
two facts after proving that $\hat{\phi}_m$ has certain dynamical
properties.

\begin{lem}
  \label{lem:convergence_to_symmetric_simplex}
  For each $H \in \mc{H}$, there exists an open set
  $\hat{C}_H \subset \symflags(\Sym^mV)$ containing
  $\hat{\Lambda}_m - \dee \hat{S}_H$, such that for any infinite
  sequence $h_n \in H$ and $\xi \in \hat{C}_H$, we have
  \[
    \rho^m(h_n)\xi \to \dee \hat{S}_H.
  \]
\end{lem}
\begin{proof}
  For each $H \in \mc{H}$, we let $C_H \subset \P(\Sym^mV)$ be the set
  coming from \Cref{prop:simplex_connected_comp}. Let $h_n$ be a
  divergent sequence in some $H \in \mc{H}$, and let $H_0$ be a
  finite-index free abelian
  subgroup. \Cref{cor:symmetric_simplex_attracting_subspace} says that
  some face $F$ of $S_H$ spans an attracting subspace for
  $\rho^m(h_n)$. The corresponding repelling subspace is a direct sum
  of weight spaces for the restriction $\rho^m|_{H_0}$, so it is
  contained in $\hyp{w^m}$ for a vertex $w^m$ of the dual simplex
  $S_H^*$. So, for any $x \in C_H$, any subsequence of $\rho^m(h_n)x$
  subconverges to a point in $[\spn(F)]$. In fact, $\rho^m(h_n)x$
  subconverges to a point in $\overline{F} \subset \dee S_H$, since
  $C_H$ is $\rho^m(H)$-invariant and
  $C_H \cap \supp(F) = \overline{F}$.

  Then, we can dually define a set $C_H^* \subset \P(\Sym^mV^*)$, and
  take
  \[
    \hat{C}_H = \{(x,w) \in \symflags(\Sym^mV) : x \in C_H, w \in C_H^*\}.
  \]
\end{proof}

\subsection{Continuity and compactness}

\begin{lem}
  The set $\hat{\Lambda}_m$ is closed.
\end{lem}
\begin{proof}
  Let $(x_n, w_n)$ be a sequence in $\hat{\Lambda}_m$, and let
  $z_n = \hat{\phi}_m(x_n, w_n)$. Up to subsequence, $z_n$ converges
  to $z \in \bgamh$.

  If $z$ is a conical limit point, let $\gamma_n$ be a sequence
  limiting conically to $z$, chosen so that for any $z' \ne z$, we
  have $\gamma_n^{-1}z' \to b$ and $\lim \gamma_n^{-1}z_n = a \ne b$.

  Then $\hat{\phi}^{-1}(\gamma_n^{-1}z_n)$ converges to
  $\hat{\phi}^{-1}(a)$, and thus $\hat{\phi}^{-1}(\gamma_n^{-1}z_n)$
  lies in a fixed compact subset $X$ of
  $\Opp(\hat{\phi}^{-1}(b)) \cap \hat{\Lambda}_\Omega(\Gamma)$. By
  definition, this means that for every $n$,
  $\hat{\phi}_m^{-1}(\gamma_n^{-1}z_n)$ lies in the set
  \[
    \iota(X) \cup \bigcup_{H' \in \mc{H}_X}S_{H'},
  \]
  Arguing as in \Cref{prop:simplex_connected_comp}, we see that this
  set is compact. So by antipodality of $\hat{\phi}_m$, the sets of
  flags $\hat{\phi}_m^{-1}(\gamma_n^{-1}z_n)$ lie in a fixed compact
  subset of
  $\Opp(\hat{\phi}_m^{-1}(b)) = \Opp(\hat{\iota}(\hat{\phi}^{-1}(b)))$
  and by \Cref{prop:symmetric_power_dynamics},
  \[
    (x_n, w_n) \in \rho^m(\gamma_n)\hat{\phi}_m^{-1}(\gamma_n^{-1}z_n)
  \]
  converges to $\hat{\iota}(\hat{\phi}^{-1}(z))$.
  
  If $z$ is a parabolic point, we let $H = \Stab_\Gamma(z)$, and
  choose $h_n \in H$ so that $h_n^{-1}z_n \in K$ for a fixed compact
  $K - \{z\}$. By \Cref{prop:simplex_connected_comp}, we know that for
  all $n$, $\hat{\phi}_m^{-1}(h_n^{-1}z_n)$ lies in a fixed compact
  subset of $\hat{C}_H$. Then,
  \Cref{lem:convergence_to_symmetric_simplex} implies that
  \[
    (x_n, w_n) \in \rho^m(h_n)\hat{\phi}_m^{-1}(h_n^{-1}z_n)
  \]
  subconverges to a point in $\dee \hat{S}_H$.
\end{proof}

\begin{prop}
  The map $\hat{\phi}_m$ is continuous.
\end{prop}
\begin{proof}
  Let $(x_n, w_n)$ be a sequence in $\hat{\Lambda}_m$, converging to
  $(x, w)$ (which we know lies in $\hat{\Lambda}_m$ by the previous
  proposition). Let $z_n = \hat{\phi}_m(x_n, w_n)$, and suppose for a
  contradiction that up to subsequence $z_n \to z$ for
  $z \ne \hat{\phi}_m(x, w)$.

  \Cref{prop:simplex_connected_comp} then implies that $z_n$ lies in a
  compact subset $K \subset \bgamh$ so that the closure of
  $\hat{\phi}_m^{-1}(K)$ is opposite to $(x,w)$. This contradicts the
  fact that $(x_n, w_n)$ converges to $(x,w)$.
\end{proof}

At this point, we have shown that
$\hat{\phi}_m:\hat{\Lambda}_m \to \bgamh$ is a continuous equivariant
surjective transverse map, and that $\hat{\Lambda}_m$ is a compact
subset of $\symflags(\Sym^mV)$. So, we can finish the proof of
\Cref{thm:symmetric_powers_egf} by showing:
\begin{prop}
  The map $\hat{\phi}_m:\hat{\Lambda}_m \to \bgamh$ extends the
  convergence action of $\Gamma$ on $\bgamh$.
\end{prop}
\begin{proof}
  We apply \Cref{prop:conical_peripheral_implies_egf}. If $\gamma_n$
  is a sequence limiting conically to $z$, then the results of
  \Cref{sec:convex_cocompact} imply that the sequences
  $\rho(\gamma_n^\pm)$ have unique limit points in $\symflags(V)$, all
  lying in $\hat{\Lambda}_\Omega(\Gamma)$. Since
  $\hat{\iota}(\hat{\Lambda}_\Omega(\Gamma))$ is a subset of
  $\hat{\Lambda}_m$, \Cref{prop:symmetric_power_dynamics} ensures that
  the first condition of \Cref{prop:conical_peripheral_implies_egf}
  holds.

  On the other hand, for each parabolic point $p$, we take $\hat{C}_p$
  to be the open set $\hat{C}_H$ considered in
  \Cref{lem:convergence_to_symmetric_simplex}, for
  $H = \Stab_\Gamma(p)$. \Cref{lem:convergence_to_symmetric_simplex}
  implies that $\hat{C}_p$ contains
  $\hat{\Lambda}_m - \hat{\phi}_m^{-1}(p)$ and that for any
  $(x, w) \in \hat{C}_p$ and any infinite sequence $h_n \in H$,
  $\rho^m(h_n)\xi$ subconverges to a point in $\dee \hat{S}_H$.
\end{proof}

\subsection{Stability}

We have now shown that the representations $\rho^m$ are all extended
geometrically finite. Our last goal for the section is the following
(which implies \Cref{thm:symmetric_powers_stable}):
\begin{prop}
  \label{prop:sym_power_peripheral_stab}
  The space $\Hom(\Gamma, \SL(\Sym^mV))$ is peripherally stable with
  respect to $(\rho^m, \hat{\phi}_m)$.
\end{prop}

The main step in the proof is the following:
\begin{lem}
  \label{lem:symmetric_power_perturb}
  Let $H_0$ be a finite-index free abelian subgroup of some
  $H \in \mc{H}$, with $H = \Stab_\Gamma(p)$. For any open set
  $U \subset \P(\Sym^mV)$ containing $S_H$ and any compact
  $K \subset C_p$, there exists a cofinite subset $T \subset H_0$ and
  an open set $\mc{W} \subset \Hom(\Gamma, \SL(\Sym^mV))$ containing
  $\rho^m$ such that for any $\sigma \in \mc{W}$, we have
  $\sigma(h)K \subset U$ for any $h \in T$.
\end{lem}
\begin{proof}
  We fix $U$ and $K$ as in the proposition, and proceed by
  contradiction. So, suppose that there exists a sequence of distinct
  group elements $h_n \in H_0$, a sequence of representations
  $\sigma_n:\Gamma \to \SL(\Sym^mV)$, and a sequence of points
  $x_n \in K$ such that $\sigma_n \to \rho^m$ and
  $\sigma_n(h_n)x_n \notin U$. Up to subsequence we can assume that
  $x_n$ converges to some $x \in K \subset C_p$. We let $\Phi^m$
  denote the set of generalized weights of $\rho^m|_{H_0}$, and we let
  $\Phi^m_n$ denote the generalized weights of $\sigma_n|_{H_0}$.

  We choose a norm $|\cdot|$ on $H_0 \ten_\Z \R$. Then up to
  subsequence $h_n / |h_n|$ converges to $h_\infty \in H_0 \ten_\Z \R$
  with $|h_\infty| = 1$.

  Since $\Phi^m$ is finite, there is a face $\tilde{F}$ of the
  $k$-simplex $\mc{C}(\rho^m|_{H_0})$ and a constant $M > 0$, such
  that for every weight $\mu \in \Phi^m(F) = \Phi^m \cap \tilde{F}$,
  and every weight $\mu^\opp \in \Phi^m - \Phi(F)$, we have
  \[
    \mu(h_\infty) - \mu^\opp(h_\infty) > M.
  \]
  We let $V^m_F \subset \Sym^mV$ denote the span of the weight spaces
  of the weights in $\Phi^m(F)$; by definition $\P(V^m_F)$ is the
  projective span of a face of $S_H$.
  
  \Cref{prop:weights_deform_continuously} implies that as a set with
  multiplicity, the weights $\Phi_n^m$ converge to the weights
  $\Phi^m$. So, for each $n$, there is a subset
  $\theta_n \subset \Phi_n^m$ such that $\theta_n$ converges to
  $\Phi^m \cap \dee \mc{C}(\rho^m|_{H_0})$, and a subset
  $\theta_n(F) \subset \theta_n$ such that $\theta_n(F)$ converges to
  $\Phi^m(F)$. \Cref{prop:weights_deform_continuously} also implies
  that for sufficiently large $n$, all of the weights in $\theta_n$
  must have one-dimensional generalized weight spaces, converging to
  the vertices of $S_H$.

  This means that for each $n$, there are simplices $S_H^n$ and
  $(S_H^n)^*$, invariant under the action of $\sigma_n(H)$, such that
  $S_H^n \to S_H$ and $(S_H^n)^* \to S_H^*$. So, we can find a
  sequence of group elements $g_n \in \SL(V)$, with $g_n$ converging
  to the identity, so that $\sigma_n' = g_n \sigma_n g_n^{-1}$
  preserves the simplices $S_H$ and $S_H^*$. Moreover, the vertices of
  $S_H$ are the weight spaces $V_\mu$ of $\sigma_n'$ for
  $\mu \in \theta_n$, and the space $V^m_F$ is spanned by weight
  spaces $V_{\mu(F)}$ of $\sigma_n'$ for $\mu(F) \in \theta_n(F)$. We
  can also assume that $\sigma_n'$ preserves the complementary weight
  space $(V_F^m)^\opp$ for $V_F$.

  Now, since $\theta_n(F)$ converges to $\Phi^m(F)$, for sufficiently
  large $n$ we must have
  \[
    \mu_n(h_\infty) - \mu_n^\opp(h_\infty) > M
  \]
  for every $\mu_n \in \theta_n(F)$ and every
  $\mu_n^\opp \in \Phi^m_n - \theta_n(F)$. This also means that for
  sufficiently large $n$ we have
  \[
    \mu_n(h_n) - \mu_n^\opp(h_n) > M|h_n|.
  \]
  Then, letting $r^+(g)$ and $r^-(g)$ respectively denote the largest
  and smallest modulus of any eigenvalue of $g$, we see that for
  sufficiently large $n$,
  \[
    \frac{r^-(\sigma_n'(h_n)|_{V_F^m})}{r^+(\sigma_n'(h_n)|_{(V_F^m)^\opp})}
    > \exp(M |h_n|).
  \]
  We may choose an inner product on $\Sym^mV$ so that $V_F^m$ and
  $(V_F^m)^\opp$ are orthogonal. And, up to change-of-basis lying in
  compact subset of $\SL(V_F^m \oplus (V_F^m)^\opp)$, the restriction
  of (the complexifications of) $\sigma_n'(H_0)$ to $V_F^m$ and
  $(V_F^m)^\opp$ are both upper-triangular. Then we can apply
  \Cref{lem:upper_triangular_norm_bound} to see that the ratio
  \[
    \frac{\mathbf{m}(\sigma_n'(h_n)|_{V_F^m})}{||\sigma_n'(h_n)|_{(V_F^m)^\opp||}}
  \]
  tends to infinity as $n \to \infty$. Then by
  \Cref{lem:norm_attract_estimate}, for sufficiently large $n$,
  $\sigma_n'(h_n)x_n$ lies in a small neighborhood of
  $\P(V_F)$. Moreover, we know that the set $C_p$ is
  $\sigma_n'(H)$-invariant, since the simplex $S_H^*$ is
  $\sigma_n'(H)$-invariant. Since $x$ lies in $C_p$, $\sigma_n'(h_n)x$
  lies in an arbitrarily small neighborhood of $\P(V_F) \cap C_p$. By
  definition this intersection is a face of $S_H$, so for large enough
  $n$, $\sigma_n(h_n)x$ must lie in an arbitrarily small neighborhood
  of this face, giving a contradiction.
\end{proof}

\begin{proof}[Proof of \Cref{prop:sym_power_peripheral_stab}]
  We want to show that if $H = \Stab_\Gamma(p)$ for a parabolic point
  $p$, $K$ is a compact subset of $\hat{C}_p$, $U$ is a
  neighborhood of $\hat{S}_H$, and $T \subset H$ is a cofinite
  subset such that
  \begin{equation}
    \label{eq:cofinite_subset_in_nbhd}
    \rho^m (T) \cdot K \subset U,
  \end{equation}
  then we can find a neighborhood $\mc{W}$ of $\rho^m|_H$ in
  $\Hom(H, \SL(\Sym^mV))$ so that for any $\sigma \in \mc{W}$,
  \begin{equation}
    \label{eq:cofinite_subset_deformation}
    \sigma(T) \cdot K \subset U.
  \end{equation}

  For simplicity, we will not work in the space of flags
  $\symflags(\Sym^mV)$. Instead we will just show that that if
  (\ref{eq:cofinite_subset_in_nbhd}) holds for a compact
  $K \subset C_p \subset \P(\Sym^mV)$ and an open neighborhood $U$ of
  $S_H$ in $\P(\Sym^mV)$, then (\ref{eq:cofinite_subset_deformation})
  holds also.

  Fix a finite-index free abelian subgroup $H_0 \subseteq H$. It
  suffices to show that we can choose an open
  $\mc{W} \subset \Hom(\Gamma, \SL(\Sym^mV))$ so that
  \[
    \sigma(T \cap H_0) \cdot K \subset U
  \]
  for all $\sigma \in \mc{W}$. It follows immediately from
  \Cref{lem:symmetric_power_perturb} that we can find a cofinite set
  $T' \subset H_0$ and an open set
  $\mc{W} \subset \Hom(\Gamma, \SL(\Sym^mV))$ so that for all
  $\sigma \in \mc{W}'$, we have
  \[
    \sigma(T') \cdot K \subset U.
  \]
  But then since $(T \cap H_0) - T'$ is finite, we can just shrink
  $\mc{W}$ to get the desired result.
\end{proof}


\printbibliography

\end{document}